\documentclass[graybox]{svmult}
\smartqed

\usepackage{mathptmx}       
\usepackage{helvet}        
\usepackage{courier}        
\usepackage{type1cm}        
\usepackage{makeidx}         
\usepackage{graphicx}       
\usepackage[bottom]{footmisc}
\usepackage{amsmath,amssymb,amscd}
\usepackage{mathrsfs}
\usepackage{color}
\usepackage{bm}

\makeindex 

\newcommand{\DD}{\mathscr{D}}
\newcommand{\C}{\mathscr{C}}

\newcommand{\V}{\mathscr{V}}
\newcommand{\HHH}{\mathscr{H}}
\newcommand{\N}{\mathbb{N}}
\newcommand{\K}{\mathbb{K}}
\newcommand{\R}{\mathbb{R}}
\newcommand{\X}{\mathbb{X}}
\newcommand{\Y}{\mathscr{Y}}

\newcommand{\bs}{\boldsymbol}
\newcommand{\eps}{\varepsilon}
\newcommand{\Rot}{\nabla\!\times\!}
\newcommand{\lraup}{\relbar\joinrel\rightharpoonup}

\DeclareMathOperator{\supp}{supp}

\begin{document}
	
\title*{Variational and Quasi-Variational Inequalities with Gradient Type Constraints}
\author{Jos\'e Francisco Rodrigues\thanks{JFR acknowledges the hospitality of the Weierstrass Institute in Berlin (WIAS) during a visit when part of this work was developed.} and Lisa Santos}
\institute{Jos\'e Francisco Rodrigues \at CMAFcIO and Departamento de Matem\'atica, Faculdade de Ci\^encias, Universidade de Lisboa, P-1749-016 Lisboa,
Portugal, \email{jfrodrigues@ciencias.ulisboa.pt}
\and Lisa Santos \at CMAT and Departamento de Matem\'atica, Escola de Ci\^encias, Universidade do Minho, Campus de Gualtar,
4710-057 Braga, Portugal, \email{lisa@math.uminho.pt}}

\maketitle
	
\abstract{This survey on stationary and evolutionary problems with gradient constraints is based on developments of monotonicity and compactness methods applied to large classes of scalar and vectorial solutions to variational and quasi-variational inequalities. Motivated by models for critical state problems and applications to free boundary problems in Mechanics and in Physics, in this work several known properties are collected and presented and a few novel results and examples are found.}


\vspace{2mm}

	
\section{Introduction}

The mathematical analysis of the unilateral problems were initiated in 1964 simultaneously by Fichera, to solve the Signorini problem in elastostatics \cite{Fi64}, and by Stampacchia \cite{St64}, as an extension of the Lax-Milgram lemma with application to the obstacle problem for elliptic equations of second order. The evolution version, coining the expression variational inequalities and introducing weak solutions, was first treated in the pioneer paper of 1966 of Lions and Stampacchia \cite{LionsStamppachia1967}, immediately followed by many others, including the extension to pseudo-monotone operators by Br\'ezis in 1968 \cite{Brezis1968} (see also \cite{Lions1969}, \cite{{BC1978}}, \cite{KinderlehrerStamppachia1980} or \cite{S1997}). The importance of the new concept was soon confirmed by its versatility of their numerical approximations and in the first applications to optimal control of distributed systems in 1966-1968 by Lions and co-workers \cite{Li68} and to solve many problems involving inequalities in Mechanics and Physics, by Duvaut and Lions in their book of 1972 \cite{DuvautLions1972}, as well as several free boundary problems which can be formulated as obstacle type problems (see the books \cite{BC1978}, \cite{KinderlehrerStamppachia1980}, \cite{F82} or \cite{Rodrigues1987}).

Quasi-variational inequalities are a natural extension of the variational inequalities when the convex sets where the solutions are to be found depend on the solutions themselves. They were introduced by Bensoussan and Lions in 1973 to solve impulse control problems \cite{BenLi1982} and were developed, in particular, for certain free boundary problems, as the dam problem by Baiocchi in 1974 (see, for instance \cite{BC1978} and its references), as implicit unilateral problems of obstacle type, stationary or evolutionary \cite{MignotPuel1977},  in which the constraints are only on the solutions. 

While variational inequalities with gradient constraints appeared already to formulate the elastic-plastic torsion problem with an arbitrary cross section in the works of Lanchon, Duvaut and Ting around 1967 (see \cite{DuvautLions1972} or \cite{Rodrigues1987}, for references), the first physical problem with gradient constraints formulated with quasi-variational inequalities of evolution type were proposed for the sandpile growth in 1986 by Prighozhin, in \cite{P86} (see also \cite{Prigozhin1994}). However, only ten years later the first mathematical results appeared, first for variational inequalities, see \cite{Pri1996-1}  and the independent work \cite{AEW1996}, together with a similar one for the magnetisation of type-II superconductors \cite{Pri1996-2}. This last model has motivated a first existence result for the elliptic quasi-variational in \cite{KunzeRodrigues2000}, which included other applications in elastoplasticity and in electrostatics, and was extended to the parabolic framework for the p-Laplacian with an implicit gradient constraint in \cite{RodriguesSantos2000}. This result was later extended to quasi-variational solutions for first order quasilinear equations in \cite{RodriguesSantos2012}, always in the scalar cases, and extended recently to a more general framework in \cite{MirandaRodriguesSantos2018}.  The quasi-variational approach to the sand pile and the superconductors problems, with extensions to the simulation of lakes and rivers, have been successfully developed  also with numerical approximations (see \cite{PrigozhinZaltzman2003}, \cite{BP2006}, \cite{BP2010}, \cite{BarrettPrigozhin2013}, \cite{BP2014}, for instance). 

Although the literature on elliptic variational inequalities with gradient constraints is large and rich, including the issue of the regularity of the solution and their relations with the obstacle problem, it is out of the scope of this work to make its survey. Recent developments on stationary quasi-variational inequalities can be found in \cite{Kenmochi2007}, \cite{MirandaRodriguesSantos2009}, 
 \cite{MKK}, \cite{HintermullerRautenberg2012}, \cite{AzevedoMirandaSantos2013}, \cite{FKS2014}, \cite{KM2016}, \cite{AR2018} and the survey \cite{Kubo2017}.
 
With respect to evolutionary quasi-variational problems with gradient constraint, on one hand, Kenmochi and co-workers, in \cite{Ken2013}, \cite{FukaoKenmochi2013}, \cite{KenmochiNiezgodka2016}, \cite{Kubo2017} and \cite{KuboYamazaki2018}, have obtained interesting results by using variational evolution inclusions in Hilbert spaces with sub-differentials with a non-local dependence on parameters, and on the other hand,  Hinterm\" uller and Rautenberg in \cite{HintermullerRautenberg2013}, using the pseudo-monotonicity and the $\C^0$-semigroup approach of Br\'ezis-Lions,  in \cite{HintermullerRautenberg2017}, using contractive iteration arguments that yield uniqueness results and numerical approximations in interesting but special situations, and in \cite{HintermullerRautenbergStrogies2018}, by time semi-discretisation of a monotone in time problem, have developed interesting numerical schemes that show the potential of the quasi-variational method. Other recent results on evolutionary quasi-variational inequalities can be also found in \cite{KenmochiNiezgodka2016} and \cite{KuboYamazaki2018}, both in more abstract frameworks and oriented to unilateral type problems and, therefore, with limited interest to constraints on the derivatives of the solutions.

This work is divided into two parts on stationary and evolutionary problems, respectively. The first one, after introducing the general framework of partial differential operators of p-Laplacian type and the respective functional spaces, exposes a brief introduction to the well-posedness of elliptic variational inequalities, with precise estimates and the use of the Mosco convergence of convex sets. Next section surveys old and recent results on the Lagrange multiplier problem associated with the gradient constraint, as well as its relation with the double obstacle problem and the complementarity problem. The existence of solutions to stationary quasi-variational inequalities is presented in the two following sections, one by using a compactness argument and the Leray-Schauder principle, extending \cite{KunzeRodrigues2000}, and the other one, for a class of Lipschitz nonlocal nonlinearity, by the Banach fixed point applied to the contractive property of the variational solution map in the case of smallness of data, following an idea of \cite{HintermullerRautenberg2012}. The first part is completed with three physical problems: a nonlinear Maxwell quasi-variational inequality motivated by a superconductivity model; a thermo-elastic system for a locking material in equilibrium and an ionisation problem in electrostatics. The last two problems, although variants of examples of \cite{KunzeRodrigues2000}, are new.

The second part treats evolutionary problems, of parabolic, hyperbolic and degenerate type. The first section treats weak and strong solutions of variational inequalities with time dependent convex sets, following \cite{MirandaRodriguesSantos2018} and giving explicit estimates on the continuous dependence results. The next two sections are, respectively, dedicated to the scalar problems with gradient constraint, relating the original works \cite{Santos1991} and \cite{Santos2002} to the more recent inequality for the transport equation of \cite{RodriguesSantos2015} for the variational case, and to the scalar quasi-variational strong solutions presenting a synthesis of \cite{RodriguesSantos2000} with \cite{RodriguesSantos2012} and an extension to the linear first order problem as a new corollary. The following section, based on \cite{MirandaRodriguesSantos2018},  briefly describes the regularisation penalisation method to obtain the existence of weak solutions by compactness and monotonicity. The next section also develops the method of \cite{HintermullerRautenberg2017} in two concrete functional settings with nonlocal Lipschitz nonlinearities to obtain, under certain explicit conditions, novel results on the existence and uniqueness of strong (and weak) solutions of evolutionary quasi-variational inequalities. Finally, the last section presents also three physical problems with old and new observations, as applications of the previous results, namely on the dynamics of the sandpile of granular material, where conditions for the finite time stabilisation are described, on an evolutionary superconductivity model, in which the threshold is temperature dependent, and a variant of the Stokes flow for a thick fluid, for which it is possible to explicit conditions for the existence and uniqueness of a strong quasi-variational solution.

\section{Stationary problems}

\subsection{A general $p$-framework}

Let $\Omega$ be a bounded open subset of $\R^d$, with a Lipschitz boundary, $d\geq2$. We represent a real vector function by a bold symbol $\bs u=(u_1,\ldots,u_m)$ and we denote  the partial derivative of $u_i$ with respect to $x_j$ by $\partial_{x_j} u_i$.
Given real numbers $a,b$, we set $a\vee b=\max\{a,b\}$. 

For $1< p<\infty$, let L be a linear differential operator of order one in the form
\begin{equation} \label{L}
\text{L}:\bs V_p\to L^p\!(\Omega)^\ell \text{ such that }
(\text{L}\bs u)_i=\sum_{j=1}^d\sum_{k=1}^m\alpha_{ijk} \partial_{x_j}u_k,
\end{equation}
where  $\alpha_{ijk}\in L^\infty(\Omega)$, $i=1,\ldots,\ell$, $j=1,\ldots,d$, $k=1,\ldots,m$, with $\ell, m\in\N$, and $$\bs V_p=\big\{\bs u\in L^p\!(\Omega)^m:\text{L}\bs u\in L^p\!(\Omega)^\ell \big\}$$
is endowed with the graph norm.

We consider a Banach subspace $\X_p$ verifying

\begin{equation}\label{Xp}
\DD(\Omega)^m\subset\X_p\subset W^{1,p}(\Omega)^m\subset\bs V_p
\end{equation}
where
\begin{equation}\label{norm}
\|\bs w\|_{\X_p}=\|\text{L}\bs w\|_{L^p\!(\Omega)^\ell}
\end{equation}
is a norm in $\X_p$ equivalent to the one induced from $\bs V_p.$ In order that \eqref{norm} holds, we suppose there exists $c_p>0$ such that
\begin{eqnarray}\label{poincare}
\|\bs w\|_{L^p\!(\Omega)^m}\le c_p\|\text{L}\bs w\|_{L^p\!(\Omega)^\ell}\qquad \forall\bs w\in \bs V_p.
\end{eqnarray}

To fix ideas, here  the framework \eqref{L} for the operator L  can be regarded as any one of the following cases:
\begin{example} $ $\vspace{0.5mm}

$\mbox{ L} u=\nabla u$ (gradient of $u$), $m=1$, $\ell=d$;\vspace{0.5mm}

$\mbox{ L}\bs u=\Rot\bs u$ (curl of $\bs u$), $m=\ell=d=3$;\vspace{0.5mm}

$\mbox{ L}\bs u=D\bs u = \frac12(\nabla\bs u + \nabla\bs u^T)$ (symmetrised gradient of $\bs u$), $m=d$ and $\ell=d^2$.
\end{example}

When $\text{L}u=\nabla u$, we consider 
$$\X_p=W^{1,p}_0(\Omega)\quad\text{ and }\quad
\|u\|_{\X_p}=\|\nabla u\|_{L^p\!(\Omega)^d}$$ 
is equivalent to the $\bs V_p=W^{1,p}(\Omega)$ norm, by Poincar\'e inequality.

In the case $\text{L}\bs u=\nabla\times\bs u$, for a simply connected domain $\Omega$, the vector space $\X_p$ may be
\begin{equation}
\label{rot_normal0}
\X_p=\big\{\bs w\in L^p\!(\Omega)^3:\nabla\times\bs w\in L^p\!(\Omega)^3,\, \nabla\cdot\bs w=0,\, \bs w\cdot\bs n_{|_{\partial\Omega}}=0\big\},
\end{equation}
or
\begin{equation}
\label{rot_tangencial0}\X_p=\big\{\bs w\in L^p\!(\Omega)^3:\nabla\times\bs w\in L^p\!(\Omega)^3,\, \nabla\cdot\bs w=0,\, \bs w\times\bs n_{|_{\partial\Omega}}=\bs 0\big\},
\end{equation}
corresponding to different boundary conditions, where $\nabla\cdot\bs w$ means the divergence of $\bs w$. Both spaces are closed subspaces of $W^{1,p}(\Omega)^3$ and a Poincar\'e type inequality is satisfied in $\X_p$ (for details see \cite{AmroucheSeloula2010}).

When $\text{L}\bs u=D\bs u$,  we may have
$$\X_p=W^{1,p}_0(\Omega)^d\quad\text{ or }\quad\X_p =W^{1,p}_{0,\sigma}(\Omega)^d=\big\{\bs w\in W^{1,p}_0(\Omega)^d: \nabla\cdot\bs w=0\big\}$$
and $\|D\bs w\|_{L^p\!(\Omega)^{d^2}}$ is equivalent to the norm induced from $W^{1,p}(\Omega)^d$ by Poincar\'e and Korn's inequalities.

Given $\nu>0$, we introduce 
\begin{equation}\label{infinito_nu}
L^\infty_\nu(\Omega)=\big\{w\in L^\infty(\Omega): w\ge\nu\big\}.
\end{equation} 
For $G:\X_p\rightarrow L^\infty_\nu(\Omega)$, we define the nonempty closed convex set
\begin{equation}\label{kapa}
\K_{G[\bs u]}=\big\{\bs w\in\X_p:|\text{L}\bs w|\le G[\bs u]\big\},
\end{equation}
where $|\cdot|$ is the Euclidean norm in $\R^\ell$ and we denote, for $\bs w\in\bs V_p$,
\begin{equation}\label{L_cortado}
\L_p\bs u=|\text{L}\bs w|^{p-2}\text{L}\bs w.
\end{equation}
We may associate with $\L_p$ a strongly monotone operator, and there exist  positive constants $d_p$ such that for all $\bs w_1,\bs w_2\in\bs V_p$
\begin{multline}\label{monotone}
\int_\Omega\big(\L_p\bs w_1-\L_p\bs w_2\big)\cdot \text{L}(\bs w_1-\bs w_2)\\
\ge\begin{cases}d_p	\displaystyle\int_\Omega|\text{L}(\bs w_1-\bs w_2)|^p&\text{ if }p\ge2,\vspace{1mm}\\
d_p\displaystyle\int_\Omega\big(|\text{L}\bs w_1|+|\text{L}\bs w_2|\big)^{p-2}|\text{L}(\bs w_1-\bs w_2)|^2&\text{ if }1\le p<2.
\end{cases}
\end{multline}

For  $1< p< \infty$ and $\bs f\in L^1(\Omega)^m$,  we shall consider the quasi-variational inequality
\begin{equation}\label{iqv}
\bs u\in \K_{G[\bs u]}:\quad\int_{\Omega} \L_p\bs u\cdot \text{L}(\bs w-\bs u)\ge\displaystyle\int_{\Omega}\bs f\cdot(\bs w-\bs u)\quad\forall\,\bs w\in\K_{G[\bs u]}.
\end{equation}

\subsection{Well-posedness of the variational inequality}

For $g\in L^\infty_\nu(\Omega)$, it is well-know that the variational inequality, which is obtained by taking $G[\bs u]\equiv g$ in \eqref{kapa} and in \eqref{iqv}, 
\begin{equation}\label{iv}
\bs u\in \K_{g}:\quad\int_{\Omega} \L_p\bs u\cdot \text{L}(\bs w-\bs u)\ge\displaystyle\int_{\Omega}\bs f\cdot(\bs w-\bs u)\quad\forall\,\bs w\in\K_{g},
\end{equation}
has a unique solution (see, for instance, \cite{Lions1969} or \cite{KinderlehrerStamppachia1980}). The solution is, in fact, H\"older continuous on $\overline\Omega$ by recalling the (compact) Sobolev imbeddings
\begin{equation}\label{sobolev}
W^{1,p}(\Omega)\hookrightarrow\begin{cases} L^q(\Omega)\ \text{ for }1\le q< \frac{dp}{d-p}, \ \ \text{ if }p<d\vspace{1mm}\\
L^r(\Omega)\ \text{ for }1\le r<\infty\ \ \text{ if }p=d,\vspace{1mm}\\
\C^{0,\alpha}(\overline\Omega)\ \text{ for }0\le \alpha<1-\frac{d}{p}\ \ \text{ if }p>d.
\end{cases}
\end{equation}

Indeed, in the three examples above we have, for any $p>d$ and $0\le\alpha<1-\frac{d}{p}$,
\begin{equation}
\label{holder}
\K_g\subset W^{1,p}(\Omega)^m\subset \C^{0,\alpha}(\overline\Omega)^m.
\end{equation}

We note that, even if L$\bs u$ is bounded in $\Omega$, in general, this does not imply that the solution $\bs u$ of \eqref{iv} is Lipschitz continuous. However, this holds, for instance, not only in the scalar case L$=\nabla$, but, more generally if in \eqref{L} $m=1$ and $\alpha_{ij}=\eta_i\delta_{ij}$ with $\eta_i\in L_\nu^\infty(\Omega)$, $i=1,\ldots,d$ and $\delta_{ij}$  the Kronecker symbol.

We present now two continuous dependence results on the data. In particular, when \eqref{holder} holds, any solution to \eqref{iv} or \eqref{iqv} is a priori continuously bounded and therefore we could take not only $\bs f\in L^1(\Omega)^m$ but also $\bs f$ in the space of Radon measures.

\begin{theorem} \label{fest} Under the framework \eqref{L},  \eqref{Xp} and \eqref{norm} let $\bs f_1$ and $\bs f_2$ belong to $L^1(\Omega)^m$ and $g\in L^\infty_\nu(\Omega)$. Denote by $\bs u_i$, $i=1,2$, the solutions of the variational inequality \eqref{iv} with data $(\bs f_i,g)$. Then
\begin{equation}\label{f1f2_1}
\|\bs u_1-\bs u_2\|_{\X_p}\le C\|\bs f_1-\bs f_2\|_{L^1(\Omega)^m}^{\frac1{p\vee 2}},
\end{equation}
being $C$ a positive constant depending on $p$, $\Omega$ and $\|g\|_{L^\infty(\Omega)}$.
\end{theorem}
\begin{proof}
We use $\bs u_2$ as test function in the variational inequality \eqref{iv} for $\bs u_1$ and reciprocally, obtaining, after summation,
$$\int_\Omega \big(\L_p\bs u_1-\L_p\bs u_2\big)\cdot \text{L}(\bs u_1-\bs u_2)\le\int_\Omega(\bs f_1-\bs f_2)\cdot(\bs u_1-\bs u_2).$$

For $p\ge 2$, using \eqref{monotone}, since $\bs u_i\in L^\infty(\Omega)^m$,  we have
$$\|\bs u_1-\bs u_2\|_{\X_p}\le C\|\bs f_1-\bs f_2\|_{L^1(\Omega)^m}^\frac1p.$$
If $1\le p<2$, using \eqref{monotone} and $|\text{L}\bs u_i|\le M$, where $M=\|g\|_{L^\infty(\Omega)}$, we have first
$$d_p\big(2M\big)^{p-2}\int_\Omega|\text{L}(\bs u_1-\bs u_2)|^2\le\int_\Omega(\bs f_1-\bs f_2)\cdot(\bs u_1-\bs u_2)$$
and then, with $\omega_p=|\Omega|^\frac{2-p}{2p}$,
$$\|\bs u_1-\bs u_2\|_{\X_p}\le \omega_p\|\bs u_1-\bs u_2\|_{\X_2}\le C|\bs f_1-\bs f_2\|_{L^1(\Omega)^m}^\frac12,$$
concluding the proof.
\qed \end{proof}

\begin{remark} Since $|\text{L}\bs u_i|\le M$ we can always extend \eqref{f1f2_1} for any $r>d$, obtaining for some positive constants $C_\alpha>0, C_r>0$ and $\alpha=1-\frac{d}{r}>0$,
$$\|\bs u_1-\bs u_2\|_{\C^\alpha(\overline\Omega)^m}\le C_\alpha \|\bs u_1-\bs u_2\|_{\X_r}\le C_r\|\bs f_1-\bs f_2\|_{L^1(\Omega)^m}^\frac1r.$$
Indeed, it is sufficient to use the Sobolev imbedding and to observe that, for $r>p$,
$$\int_\Omega|\text{L}(\bs u_1-\bs u_2)|^r\le(2M)^{r-p}\int_\Omega|\text{L}(\bs u_1-\bs u_2)|^p.$$
\end{remark}

\begin{theorem}\label{gest}  Under the framework \eqref{L},  \eqref{Xp} and \eqref{norm} let $\bs f\in L^{1}(\Omega)^m$ and $g_1, g_2\in L^\infty_\nu(\Omega)$. Denote by $\bs u_i$, $i=1,2$, the solutions of the variational inequality \eqref{iv} with data $(\bs f,g_i)$. Then
\begin{equation}\label{cnu}
\|\bs u_1-\bs u_2\|_{\X_p}\le C_\nu\|g_1-g_2\|_{L^\infty(\Omega)}^\frac{1}{p\vee2}.
\end{equation}
\end{theorem}
\begin{proof}
Calling $\beta=\|g_i-g_j\|_{L^\infty(\Omega)}$, $i,j\in\{1,2\}$, $i\neq j$, then
$$\bs u_{i_j}=\frac{\nu}{\nu+\beta}\bs u_i\in\K_{g_j},$$
and $\bs u_{i_j}$ can be used as test function in the variational inequality \eqref{iv} satisfied by $\bs u_i$, obtaining
$$\int_\Omega\big(\L_p\bs u_1-\L_p\bs u_2\big)\cdot L(\bs u_1-\bs u_2)\le\int_\Omega \text{L}\bs u_1\cdot L(\bs u_{2_1}-\bs u_2)+
\int_\Omega \text{L}\bs u_2\cdot L(\bs u_{1_2}-\bs u_1).$$
But
$$|\text{L}(\bs u_{i_j}-\bs u_i)|=\frac{\beta}{\nu+\beta}|\text{L}\bs u_i|\le\frac{\beta M}{\nu},$$
where $M=\max\{\|g_1\|_{L^\infty(\Omega)},\|g_2\|_{L^\infty(\Omega)}\}$ and the conclusion follows.
\qed \end{proof}

We can also consider a degenerate case, by letting $\delta\rightarrow0$ in
\begin{equation}
\label{ivdelta0}
\bs u^\delta\in\K_g:\quad\delta\int_{\Omega}\L_p\bs u^\delta\cdot\text{L}(\bs w-\bs u^\delta)\ge\int_{\Omega}\bs f\cdot(\bs w-\bs u^\delta)\quad\forall\bs v\in\K_g.
\end{equation}
Indeed, since $\|L\bs u^\delta\|_{L^\infty(\Omega)}\le M$, where $M=\|g\|_{L^\infty(\Omega)}$, independently of $0<\delta\le1$, we can extract a subsequence
$$\bs u^\delta\underset{\delta\rightarrow0}{\lraup}\bs u^0\quad\text{ in }\X_p\text{-weak}$$
for some $\bs u^0\in\K_g$. Then, we can pass to the limit in \eqref{ivdelta0} and we may state:
\begin{theorem} Under the framework \eqref{L},  \eqref{Xp} and \eqref{norm}, for any $\bs f\in L^1(\Omega)^m$, there exists at least a solution $\bs u^0$ to the problem
\begin{equation}\label{ivdeg}
\bs u\in\K_g:\quad 0\ge\int_\Omega\bs f\cdot(\bs w-\bs u)\quad\forall\bs w\in\K_g.
\end{equation}\qed
\end{theorem}

In general, the strict positivity condition on the threshold $g=g(x)$, which is included in \eqref{infinito_nu}, is necessary in many interesting results, as the continuous dependence result \eqref{cnu}, which can also be obtained in a weaker form by using the Mosco convergence and observing that
$$\K_{g_n}\stackrel{M\ }{\underset{n}{\longrightarrow}}\K_g\quad \text{ is implied by }g_n\underset{n}{\longrightarrow}g\ \text{ in }L^\infty(\Omega).$$

We recall that $\K_{g_n}\stackrel{M\ }{\underset{n}{\longrightarrow}}\K_g$ iff i) for any sequences $\K_{g_n}\ni w_n\underset{n}{\lraup}w$ in $\X_p$-weak, then $w\in\K_g$ and ii) for any $w\in\K_g$ there exists $w_n\in\K_{g_n}$ such that $w_n\underset{n}{\longrightarrow}w$ in $\X_p$.

However, the particular structure of the scalar case L=$\nabla$ in $\X_p=W^{1,p}_0(\Omega)$, i.e., with $\L_p=\nabla_{\!\!\!p}\, $ and a Mosco convergence result of \cite{AzevedoSantos2004} allows us to extend the continuous dependence of the solutions of the variational inequality with nonnegative continuous gradient constraints, as an interesting result of Mosco type (see \cite{Mosco1969}).

\begin{theorem} Let $\Omega$ be an open domain with a $\C^2$ boundary, {\em L}$=\nabla$, $f\in L^{p'}\!\!(\Omega)$ and $g_\infty,g_n\in\C(\overline\Omega)$, with $g_n\ge0$ for $n\in\N$ and $n=\infty$. If $u_n$ denotes the unique solution to
\begin{equation}\label{mosco0}
u_n\in \K_{g_n}:\quad\int_{\Omega} \nabla_{\!\!\!p}\,  u_n\cdot \nabla( w-u_n)\ge\displaystyle\int_{\Omega}f\cdot(w-u_n)\quad\forall\,\ w\in\K_{g_n}
\end{equation}
then, as $n\rightarrow\infty$, $g_n\underset{n}{\longrightarrow}g_\infty$ in $\C(\overline\Omega)$ implies $u_n\underset{n}{\longrightarrow}u_\infty$ in $W^{1,p}_0(\Omega)$.
\end{theorem}
\begin{proof} By Theorem 3.12 of \cite{AzevedoSantos2004}, we have $\K_{g_n}\stackrel{M\ }{\underset{n}{\longrightarrow}}\K_{g_\infty}$. Since $|\nabla u_n|\le g_n$ in $\Omega$, we have $\|u_n\|_{W^{1,p}_0(\Omega)}\le C|\Omega|^\frac1p\|g_n\|_{\C(\overline\Omega)}\le M$ independently of $n$ and, therefore, we may take a subsequence $u_n\underset{n}{\lraup}u_*$ in $W^{1,p}_0(\Omega)$.
Then $u_*\in\K_{g_\infty}$. For any $w_\infty\in\K_{g_\infty}$, take $w_n\in\K_{g_n}$ with $w_n\underset{n}{\longrightarrow}w_\infty$ in $W^{1,p}_0(\Omega)$ and, using Minty's Lemma and letting $n\rightarrow\infty$ in
$$\int_{\Omega}\nabla_{\!\!\!p}\,  w_n\cdot\nabla(w_n-u_n)\ge \int_{\Omega}f(w_n-u_n)$$
we conclude that $u_*=u_\infty$ is the unique solution of \eqref{mosco0} for $n=\infty$.
The strong convergence follows easily, by choosing $v_n\underset{n}{\longrightarrow}u_\infty$ with $v_n\in\K_{g_n}$, from
$$\int_{\Omega}|\nabla(u_n-u_\infty)|^p\le\int_{\Omega}f(u_n-v_n)+\int_{\Omega}\nabla_{\!\!\!p}\, u_n\cdot\nabla(v_n-u_\infty)-\int_{\Omega}\nabla_{\!\!\!p}\, u_\infty\cdot\nabla(u_n-u_\infty)\underset{n}{\rightarrow}0.$$
\qed\end{proof}

\subsection{Lagrange multipliers}\label{2.3}

In the special case $p=2$, $\L_2=\text{L},$ consider the variational inequality ($\delta>0$)
\begin{equation}\label{ivdelta}
\bs u^\delta\in \K_{g}:\quad\delta\int_{\Omega} \text{L}\bs u^\delta\cdot \text{L}(\bs w-\bs u^\delta)\ge\displaystyle\int_{\Omega}\bs f\cdot(\bs w-\bs u^\delta)\quad\forall\,\bs w\in\K_{g}
\end{equation}
and the related Lagrange multiplier problem, which is equivalent to the problem of finding $(\lambda^\delta,\bs u^\delta)\in\,\big( L^\infty(Q_T)^m\big)'\times\X_\infty$ such that
\begin{subequations}\label{weaklm}
\begin{equation}\label{weaklm1}
\langle\lambda^\delta \text{L}\bs u^\delta,\text{L}\bs\varphi\rangle_{(L^\infty(\Omega)^m)'\times L^\infty(\Omega)^m}= \int_\Omega \bs f\cdot\bs\varphi\quad\forall\bs\varphi\in\X_\infty,
\end{equation}
\begin{equation}\label{weaklm2}
|\text{L}\bs u^\delta|\le g\ \text{ a.e. in }\Omega,\quad\lambda^\delta\ge \delta,\quad(\lambda^\delta-\delta)(|\text{L}\bs u^\delta|-g)=0\quad \text{ in }\big(L^\infty(\Omega)^m\big)',
\end{equation}
\end{subequations}
where we set $\X_\infty=\big\{\bs \varphi\in L^2(\Omega)^m:\text{L}\bs\varphi\in L^\infty(\Omega)^\ell\big\}$ and define
$$\langle\lambda\bs\alpha,\bs\beta\rangle_{(L^\infty(\Omega)^m)'\times L^\infty(\Omega)^m}=\langle\lambda,\bs\alpha\cdot\bs\beta\rangle_{L^\infty(\Omega)'\times L^\infty(\Omega)}\quad\forall\lambda\in L^\infty(\Omega)'\ \forall\bs\alpha,\bs\beta\in L^\infty(\Omega)^m.$$
In fact, arguing as in  \cite[Theorem 1.3]{AzevedoSantos2017}, which corresponds only to the particular scalar case $\text{L}=\nabla$, we can prove the following theorem:
\begin{theorem}\label{thm_td_weak}  Suppose that $\Omega$ is a bounded open subset of $\,\R^d$ with Lipschitz boundary and   the assumptions \eqref{L} and \eqref{Xp} are satisfied with $p=2$.  Given $\bs f\in L^2(\Omega)^m$ and $g\in L^\infty_\nu(\Omega)$,
\begin{enumerate}
\item  if $\delta>0$, problem \eqref{weaklm} has a solution
$$(\lambda^\delta, u^\delta)\in L^{\infty}(\Omega)'\times  \X_\infty;$$
\item at least for a subsequence  $(\lambda^\delta, u^\delta)$ of solutions of problem \eqref{weaklm}, we have
\begin{equation*}
\lambda^\delta\underset{\delta\rightarrow0}{\lraup}\lambda^0\quad\text{in}\quad L^\infty(\Omega)',\qquad u^\delta\underset{\delta\rightarrow0}{\lraup}u^0\quad\text{in}\quad \X_\infty.
\end{equation*}
		
In addition, $\bs u^\delta$ also solves \eqref{ivdelta} for each $\delta\ge0$ and $(\lambda^0,u^0)$ solves  problem \eqref{weaklm} for $\delta=0$.
\end{enumerate}
\qed\end{theorem}

We observe that the last condition in \eqref{weaklm} on the Lagrange multiplier $\lambda^\delta$ corresponds, in the case of integrable functions, to say that a.e. in $\Omega$
\begin{equation}\label{gmm}
\lambda^\delta\in \mathscr K_\delta(|\text{L}\bs u^\delta|-g)
\end{equation}
where, for $\delta\ge0$, $\mathscr K^\delta$ is the family of maximal monotone graphs given by $\mathscr K_\delta(s)=\delta$ if $s<0$ and $\mathscr K_\delta(s)=[\delta,\infty[$ if $s=0$. In general, further properties for $\lambda^\delta$ are unknown except in the scalar case with $\text{L}=\nabla$.

The model of the elastic-plastic torsion problem corresponds to the variational inequality with gradient constraint \eqref{ivdelta} with $\delta=1$, $p=2=d$, L$=\nabla$, $g\equiv 1$ and $\bs f=\beta$, a positive constant. In \cite{Brezis1972}, Br\'ezis proved the equivalence of this variational inequality with the Lagrange multiplier problem \eqref{weaklm} with these data and assuming $\Omega$ simply connected, showing also that $\lambda\in L^\infty(\Omega)$ is unique and even continuous in the case of $\Omega$ convex. This result was  extended to multiply connected domains by Gerhardt in \cite{Gerhardt1976}. Still for $g\equiv1$, Chiad\`o Piat and Percival extended the result for more general operators in \cite{ChiadoPiatPercivale1994}, being $f\in L^r(\Omega), r>d\ge2$, proving that $\lambda$ is a Radon measure but leaving open the uniqueness.
Keeping $g\equiv1$ but assuming $\delta=0$,  problem \eqref{weaklm} is the Monge-Kantorovich mass transfer problem (see \cite{Evans1997} for details) and the convergence $\delta\rightarrow0$ in the theorem above links this problem to the limit of Lagrange multipliers for elastic-plastic torsion problems with coercive constant $\delta>0$. In \cite{DePascaleEvansPratelli2004}, for the case $\delta=0$, assuming  $\Omega$ convex and $f\in L^q(\Omega),$  $2\le q\le\infty$ with $\int_\Omega f=0$,
Pascale, Evans and Pratelli proved the existence of $\lambda^0\in L^q(\Omega)$ solving \eqref{weaklm}. In \cite{AzevedoMirandaSantos2013}, for $\Omega$ any bounded Lipschitz domain, it was proved the existence of solution $(\lambda,u)\in L^\infty(\Omega)'\times W^{1,\infty}_0(\Omega)$ of the problem \eqref{weaklm}, with $\delta=1$,  $f\in L^2(\Omega)$, $g\in W^{2,\infty}(\Omega)$ and in \cite{AzevedoSantos2017} this result was extended for $\delta\ge0$, with $f\in L^\infty(\Omega)$ and $g$ only in  $L^\infty(\Omega)$, as it is stated in the theorem above, but for  L$=\nabla$. Besides,  when $g\in\C^2(\Omega)$ and $\Delta g^2\le0$, in \cite{AzevedoSantos2017} it is also shown that
$\lambda^\delta\in L^q(\Omega)$, for any $1\le q<\infty$ and $\delta\ge0$. 

Problem \eqref{weaklm} is also related to the equilibrium of the  table sandpiles problem (see \cite{Pri1996-1}, \cite{CC2004}, \cite{DumontIgbida2009}) and other problems in the Monge-Kantorovich theory (see  \cite{Evans1997}, \cite{Ambrosio2003}, \cite{BP2006}, \cite{Ig2009}).

In the degenerate case $\delta=0$, problem \eqref{weaklm} is also associated with the limit case $p\rightarrow\infty$ of the $p-$Laplace equation and related problems to the infinity Laplacian (see, for instance \cite{BBM91} or \cite{JPR2016} and their references), as well as in some variants of the optimal transport probem, like the obstacle Monge-Kantorovich equation (see \cite{CaffareliMcCann2010}, \cite{Figalli2010} and \cite{IgbidaNguyen2018}).

There are other problems with gradient constraint that are related with the scalar variational inequality \eqref{ivdelta}  with L$=\nabla$. To simplify, we assume $\delta=1$.

When $f$ is constant and $g\equiv1$,  it is well known that the variational inequality \eqref{ivdelta} is equivalent to the two obstacles variational inequality
\begin{equation}\label{iv2ob}
u\in \K_{\underline\varphi}^{\overline\varphi}:\quad\int_{\Omega} \nabla u\cdot \nabla( w- u)\ge\displaystyle\int_{\Omega}f\cdot( w- u)\quad\forall\,w\in\K_{\underline\varphi}^{\overline\varphi},
\end{equation}
where
\begin{equation}\label{c2ob}
\K_{\underline\varphi}^{\overline\varphi}=\big\{v\in H^1_0(\Omega):\underline\varphi\le v\le \overline\varphi\big\},
\end{equation}
with $\underline\varphi(x)=-d(x,\partial\Omega)$ and $\overline\varphi(x)=d(x,\partial\Omega)$, being $d$ the usual distance if $\Omega$ is convex and the geodesic distance otherwise. This result was proved firstly by Br\'ezis and Sibony in 1971 in \cite{BrezisSibony1971}, developed by Caffarelli and Friedman in  \cite{CaffarelliFriedman1980} in the framework of elastic-plastic problems, and it was also extended in \cite{TV2000} for certain perturbations of convex functionals.

In \cite{Evans1979}, Evans proved the equivalence between \eqref{iv2ob} with the complementary problem
\eqref{compl} below, with $g=1$ However, for non constant gradient constraint, the example below shows that the problem 
\begin{equation}
\label{compl}
\max\big\{-\Delta u-f,|\nabla u|-g\big\}=0
\end{equation}
for $f,g\in L^\infty(\Omega)$  is not always equivalent to \eqref{ivdelta}, as well as the equivalence with the double obstacle variational inequality \eqref{c2ob} defined with a general constraing $g$ is not always true. We give the definition of the obstacles for $g$ nonconstant: given $x,\ z\in\overline{\Omega}$, let
\begin{multline}\label{dg}
d_g(x,z)=\inf\Big\{\displaystyle\int_0^{\delta}g(\xi(s))ds:\ \delta>0,\ \xi:[0,\delta] \rightarrow\Omega,\
\xi\mbox{ smooth },\\ 
\xi(0)=x,\ \xi(\delta)=z,
\ |\xi'|\leq 1\Big\}.
\end{multline}
This function  is a pseudometric (see \cite{Lions1982}) and  the obstacles we consider are
\begin{equation}
\label{sup}
\overline{\varphi}(x,t)=d_g(x,\partial\Omega)=\bigvee\{w(x):w\in\K_{g}\}
\end{equation}
and
\begin{equation}
\label{inf}
\underline{\varphi}(x,t)=-d_g(x,\partial\Omega)=\bigwedge\{w(x):w\in\K_{g}\}.
\end{equation}

\begin{example}	$ $
	
\noindent Let $f,g:(-1,1)\rightarrow\R$ be defined by $f(x)=2$ and $g(x)=3x^2$. Notice that $g(0)=0$ and so $g\not\in L^\infty_\nu(-1,1)$. However the solutions of the three problems under consideration exist.
The two obstacles (with respect to this	function $g$) are
\begin{equation*}
\begin{array}{lcl}
\overline{\varphi}(x)=\begin{cases}
x^3+1&\mbox{ if }x\in[-1,0[,\\ 1-x^3&\mbox{ if }x\in[0,1],
\end{cases}
&\quad \mbox{ and }\quad &\ \
\underline{\varphi}(x,t)=\begin{cases}
-x^3-1&\mbox{ if }x\in[-1,0[,\\ x^3-1&\mbox{ if }x\in[0,1].
\end{cases}
\end{array}
\end{equation*}
The function
\begin{equation*}
u(x)=\begin{cases} 1-x^2&\mbox{ if }|x|\geq\frac 23
\mbox{ and }|x|\leq 1,\\ \overline{\varphi}(x)-\frac 4{27}&\mbox{ otherwise}
\end{cases}\end{equation*}
is $\C^1$ and solves (\ref{ivdelta}) with L$=\nabla$ and $\delta=1$.
	
The function $z(x)=1-x^2$ belongs to $\K_{\underline{\varphi}}^{\overline{\varphi}}$ and, because $z''=-2$, it solves (\ref{iv2ob}).

Neither $u$ nor $z$  solve \eqref{compl}. In fact,  $-u''(x)=-6x$ in $(-\frac23,\frac23)$, so $-u''(x)\not\le2$ a.e. and $|z'|\not\le g$. \qed
\end{example}

Sufficient conditions to assure the equivalence among these problems will be given in Section \ref{evolutiva} in the framework of evolution problems.

Nevertheless, the relations between the gradient constraint problem and the double obstacle problem are relevant to study the regularity of the solution, as in the recent works of \cite{ASW2012} and \cite{CS2016}, as well as for the regularity of the free boundary in the elastic-plastic torsion problem (see \cite{F82} or \cite{Rodrigues1987} and their references). Indeed, in this case, when $g=1$ and $f=-\tau<0$ are constants, it is well-known that the elastic and the plastic regions are, respectively, given by the subsets of $\Omega\subset\R^2$
$$\big\{|\nabla u|<1\big\}=\big\{u>\underline\varphi\big\}=\big\{\lambda>1\big\}\quad\text{ and }\quad \big\{|\nabla u|=1\big\}=\big\{u=\underline\varphi\big\}=\big\{\lambda=1\big\}.$$
The free boundary is their common boundary in $\Omega$ and, by a result of Caffarelli and Rivi\`ere \cite{CaffarelliRiviere1977}, consists locally of Jordan arcs with the same smoothness as the nearest portion of $\partial\Omega$. In particular, near reentrant corners of $\partial\Omega$, the free boundary is locally analytic. As a consequence, it was observed in \cite[p.240]{Rodrigues1987} that those portions of the free boundary are stable for perturbations of data  near the reentrant corners and near the connected components of $\partial\Omega$ of nonpositive mean curvature.

Also using the equivalence with the double obstacle problem, recently, Safdari has extended some properties on the regularity and the shape of the free boundary in the case L$=\nabla$ with the pointwise gradient constraint $(\partial_{x_1}u)^q+(\partial_{x_2}u)^q\le1$, for $q>1$ (see 
\cite{Safdari2017} and its references).

\subsection{The quasi-variational solution via compactness}

We start with an existence result for the quasi-variational inequality \eqref{iqv}, following the ideas in \cite{KunzeRodrigues2000}.
\begin{theorem}\label{iqv_existence}
Under the framework \eqref{L},  \eqref{Xp} and \eqref{norm}, let $\bs f\in L^{p'}\!\!(\Omega)^m$ and $p'=\frac{p}{p-1}$. Then there exists at least one solution of the quasi-variational inequality \eqref{iqv}, provided one of the following conditions is satisfied:
\begin{enumerate}
\item the functional $G:\X_p\rightarrow L^\infty_\nu(\Omega)$  is completely continuous;
\item the functional $G:\C(\overline\Omega)^m\rightarrow L^\infty_\nu(\Omega)$  is  continuous, when $p>d$, or it satisfies also the growth condition
\begin{equation}\label{growth}
\|G[\bs u]\|_{L^r(\Omega)}\le c_0+c_1\|\bs u\|^\alpha_{L^{\sigma p}(\Omega)^m},
\end{equation}
for some constants $c_0$, $c_1\ge0$, $\alpha\ge0$, with $r>d$ and $\sigma\ge\frac1p$, when $p=d$, or $\frac1p\le\sigma\le\frac{d}{d-p}$, when $1<p<d$.
\end{enumerate}
\end{theorem}
\begin{proof}  Let $\bs u=S(\bs f,g)$ be the unique solution of the variational inequality \eqref{iv} with $g=G[\bs\varphi]$ for $\bs \varphi$ given in $\X_p$ or $\C(\overline\Omega)^m$. Since $\X_p\subset W^{1,p}(\Omega)^m$, by Sobolev embeddings, and it is always possible to take $\bs w=\bs 0$ in \eqref{iv}, we have
\begin{equation}
\label{cf}
k_s\|\bs u\|_{L^s(\Omega)^m}\le\|\bs u\|_{\X_p}\le\big(c_p\|\bs f\|_{L^{p'}\!\!(\Omega)^m}\big)^\frac1{p-1}\equiv c_{\bs f},
\end{equation}
independently of $g\in L^\infty_\nu(\Omega)$,
with $s=\frac{dp}{d-p}$ if $p<d$, for any $s<\infty$ if $p=d$, or $s=\infty$ if $p>d$, for a Sobolev constant $k_s>0$, being $c_p$ the Poincar\'e constant. By Theorem \ref{gest}, the solution map $S:L^\infty_\nu(\Omega)\ni g\mapsto\bs u\in\X_p$ is continuous.

\vspace{1mm}

\noindent{\em Case 1.} The map $T_p=S\circ G:\X_p\rightarrow\X_p$ is then also completely continuous and such that $T_p(D_{c_{\bs f}})\subset D_{c_{\bs f}}=\{\bs\varphi\in\X_p:\|\bs\varphi\|_{\X_p}\le c_{\bs f}\}$. Then, by the Schauder fixed point theorem, there exists $\bs u=T_p(\bs u)$, which solves \eqref{iv}.

\vspace{1mm}

\noindent{\em Case 2.} Set $T=S\circ G:\C(\overline\Omega)^m\rightarrow\X_p$ and $\mathscr{S} =\{\bs w\in\C(\overline\Omega)^m:\bs w=\lambda T \bs w, \lambda\in[0,1]\}$, which by \eqref{growth}  is bounded in $\C(\overline\Omega)^m$. Indeed, if $\bs w\in\mathscr S$, $\bs u=T \bs w$ solves 
\eqref{iv} with $g=G[\bs w]$ and we have, by the Sobolev's inequality,  \eqref{cf} and $\bs w=\lambda\bs u$,
\begin{align*}
\|\bs w\|_{\C(\overline\Omega)^m}&\le C\lambda\||L\bs u|\|_{L^r\!(\Omega)}\le C\|g\|_{L^r\!(\Omega)}\le C\big(c_0+c_1\|\bs w\|^\alpha_{L^{\sigma p}(\Omega)^m}\big)\\
&\le C\big(c_0+c_1k_{\sigma p}^\alpha\|\bs u\|^\alpha_{\X_p}\big)\le C\big(c_0+c_1k_{\sigma p}^\alpha c_{\bs f}^\alpha\big).
\end{align*}
Therefore $T $ is a completely continuous mapping into some closed ball of $\C(\overline\Omega)^m$ and it has a fixed point by the Leray-Schauder principle.
\qed \end{proof}

\begin{remark} The Sobolev's inequality also yields a version of Theorem \ref{iqv_existence} for $G:L^q\!(\Omega)^m\rightarrow L^\infty_\nu(\Omega)$ also merely continuous for any $q\ge1$ when $p\ge d$ and $1\le q<\frac{dq}{d-p}$ when $1<q<d$ (see \cite{KunzeRodrigues2000}).
\end{remark}

We present now examples of functionals $G$ satisfying {\em 1.} or {\em 2.} of the above theorem.

\begin{example}	$ $
	
\noindent Consider the functional $G:\X_p\rightarrow L^\infty_\nu(\Omega)$ defined as follows
$$G[\bs u](x)=F(x,\mbox{$\int$}_\Omega \bs K(x,y)\cdot\text{L}\bs u(y)dy),$$
where $F:\Omega\times\R\rightarrow\R$ is a measurable function in $x\in\Omega$ and  continuous in $ w\in\R$, satisfying $0<\nu\le F$,  and $\bs K\in\C(\overline\Omega;L^{p'}\!\!(\Omega)^\ell)$. This functional is completely continuous as a consequence of the fact that $\varphi:\X_p\rightarrow\C(\overline\Omega)$ defined by
$$w(x)=\varphi(\bs u)(x)=\int_\Omega \bs K(x,y)\cdot\text{L}\bs u(y)dy,\quad\bs u\in\X_p,\quad x\in\overline\Omega,$$
is also completely continuous. Indeed, if $\bs u_n\underset{n}{\lraup}\bs u$ in $\X_p$-weak, then $w_n\underset{n}{\longrightarrow}w$ in $\C(\overline\Omega)$, because L$\bs u_n$, being bounded in $L^p\!(\Omega)^\ell$, implies $w_n$ is uniformly bounded in $\C(\overline\Omega)$, by
$$|w_n(x)|\le\|\text{L}\bs u_n\|_{L^p\!(\Omega)^\ell}\|\bs K(x)\|_{L^{p'}\!\!(\Omega)^\ell}\le C\|\bs K\|_{\C(\overline\Omega;L^{p'}\!\!(\Omega)^\ell)}\quad\forall\,x\in\overline\Omega$$
and equicontinuous in $\overline\Omega$ by
$$|w_n(x)-w_n(z)|\le C\|\bs K(x,\cdot)-\bs K(z,\cdot)\|_{L^{p'}\!\!(\Omega)^\ell}\quad\forall\,x,z\in\overline\Omega.$$\qed
\end{example}

\begin{example}\label{4}	$ $

\noindent Let $F:\Omega\times\R^m\rightarrow\R$ be a Carath\'eodory function $F=F(x,\bs w)$, i.e., measurable in $x$ for all $\bs w\in\R^m$ and continuous in $\bs w$ for a.e. $x\in\Omega$. If, for a.e. $x\in\Omega$ and all $\bs w\in\R^m$, $F$ satisfies $0<\nu\le F(x,\bs w)$, for $p>d$ and, for $p\le d$ also
$$F(x,\bs w)\le c_0+c_1|\bs w|^\alpha,$$
for some constants $c_0, c_1\ge0$, $0\le\alpha\le\frac{p}{d-p}$ if $1<p<d$ or $\alpha\ge0$ if $p=d$, then the Nemytskii operator
$$G[\bs u](x)=F(x,\bs u(x)),\quad\text{for }\bs u\in\C(\overline\Omega)^m, \quad x\in\Omega,$$
yields a continuous functional $G:\C(\overline\Omega)^m\rightarrow L^\infty_\nu(\Omega)$, which satisfies \eqref{growth}.\qed
\end{example}

\begin{example} $ $
	
\noindent Suppose $p>d$. For fixed $g\in L^\infty_\nu(\Omega)$, defining
$$G[\bs u](x)=g(x)+\inf_{\begin{array}{c}
y\ge x\\
y\in\Omega\end{array}}|\bs u(y)|,\quad\bs u\in\C(\overline{\Omega})^m,\quad x\in\Omega,$$
where $y\ge x$ means $y_i\ge x_i$, $1\le i\le d$ (see \cite{LionsPerthame1983}),	we have an example of case {\it 2.} of Theorem \ref{iqv_existence} above.\qed
\end{example}

\subsection{The quasi-variational solution via contraction}

In the special case of ``small variations'' of the convex sets, it is possible to apply the Banach fixed point theorem, obtaining also the uniqueness of the solution to the quasi-variational inequality for $1<p\le 2$.
Here we simplify and develop the ideas of \cite{HintermullerRautenberg2012}, by starting with a sharp version of the continuous dependence result of Theorem \ref{fest} for the variational inequality \eqref{iv}.

\begin{proposition}\label{dependstat} Under the framework of Theorem \ref{fest}, let $\bs f_1,\bs f_2\in L^{p'}\!\!(\Omega)^m$, with $p'=\frac{p}{p-1}\ge 2$. Then we have
\begin{equation}\label{asymp2}
\|\bs u_1-\bs u_2\|_{\X_p}\le C_p\|\bs f_1-\bs f_2\|_{L^{p'}\!\!(\Omega)^m},\quad 1<p\le2
\end{equation}
with 
\begin{equation}\label{20}
C_2=c_2\quad\text{and}\quad C_p=(2M)^{2-p}c_p\,\tfrac{\omega_p^2}{d_p},
\end{equation}
where $c_p$ and $d_p$ are the constants, respectively, of \eqref{poincare} and \eqref{monotone},  $\omega_p=|\Omega|^{\frac{2-p}{2p}}$ and $M=\|g\|_{L^\infty(\Omega)}$.
\end{proposition}
\begin{proof}
Using \eqref{monotone} and H\"older's and Poincar\'e's inequalities, from \eqref{iv} for $\bs u_1$ with $\bs w=\bs u_2$ and for $\bs u_2$ with $\bs w=\bs u_1$, we easily obtain, first for $p=2$,
$$\|\text{L}(\bs u_1-\bs u_2)\|_{L^2\!(\Omega)^m}^2\le \|\bs f_1-\bs f_2\|_{L^{2}\!(\Omega)^m}\|\bs u_1-\bs u_2\|_{L^2\!(\Omega)^m}.$$
Hence \eqref{asymp2} follows immediately for $p=2$ since $d_2=1$. 

Observing that, for $1<p<2$, H\"older inequality yields $\|\text{L}\bs w\|_{L^{p}\!(\Omega)^\ell}\le \omega_p\|\text{L}\bs w\|_{L^2\!(\Omega)^\ell}$, using \eqref{monotone} and the H\"older inverse inequality, we get
\begin{multline*}
(2M)^{p-2}d_p\,|\Omega|^{\tfrac{p-2}p}\|\text{L}(\bs u_1-\bs u_2)\|_{L^p\!(\Omega)^\ell}^2\le\int_\Omega|\text{L}(\bs u_1-\bs u_2)|^2\big(|\text{L}\bs u_1|+|\text{L}\bs u_2|\big)^{p-2}\\
\le c_p\|\bs f_1-\bs f_2\|_{L^{p'}\!\!(\Omega)^m}\|\text{L}(\bs u_1-\bs u_2)\|_{L^p\!(\Omega)^\ell}.
\end{multline*}
and \eqref{asymp2} follows easily by recalling that $\|\bs w\|_{\X_p}=\|\text{L}\bs w\|_{L^p\!(\Omega)^\ell}$ for $\bs w\in\X_p$.
\qed
\end{proof}

We consider now a special case by separation of variables in the global constraint $G$. For $R>0$, denote 
$$D_R=\{\bs v\in\X_p:\|\bs v\|_{\X_p}\le R\}.$$
	
\begin{theorem}\label{unistat} Let $1<p\le2$, $\bs f\in L^{p'}\!\!(\Omega)^m$ and
\begin{equation}\label{gu}
G[\bs u](x)=\gamma(\bs u)\varphi(x),\quad x\in\Omega,
\end{equation}
where $\gamma:\X_p\rightarrow\R^+$ is a functional satisfying
\begin{enumerate}
\item[i)] $0<\eta(R)\le\gamma\le M(R)\quad \forall\,\bs u\in D_R$,
\item[ii)] $|\gamma(\bs u_1)-\gamma(\bs u_2)|\le\Gamma(R)\|\bs u_1-\bs u_2\|_{\X_p}\quad\ \forall\,\bs u_1,\bs u_2\in D_R$,
\end{enumerate}
for a sufficiently large $R\in\R^+$, being $\eta$, $M$ and $\Gamma$ monotone increasing positive functions of $R$,
and $\varphi\in L^\infty_\nu(\Omega)$ is given. Then, the quasi-variational inequality \eqref{iqv} has a unique solution, provided that
\begin{equation}\label{Banach_c}
\Gamma(R_{\bs f}) p\, C_p\,\|\bs f\|_{L^{p'}\!\!(\Omega)^m}<\eta(R_{\bs f}),
\end{equation}
where $C_2=c_2$ and $C_p=\big(2M(R_{\bs f})\|\varphi\|_{L^\infty(\Omega)}\big)^{2-p}c_p\,\tfrac{\omega_p^2}{d_p}$ are given as in \eqref{20}, with $R_{\bs f}=\big(c_p\|\bs f\|_{L^{p'}(\Omega)})^\frac1{p-1}$.
\end{theorem}
\begin{proof}  Let
\begin{equation*}
\begin{array}{lccl}
S: & D_R & \longrightarrow & \X_p\vspace{1mm}\\
 & \bs v & \mapsto &\bs u= S(\bs f,G[\bs v])
\end{array}
\end{equation*}
where $\bs u$ is the unique solution of the variational inequality \eqref{iv} with  $g=G[\bs v]$. 

By\eqref{cf}, any solution $\bs u$ to the variational inequality \eqref{iv} is such that $\|\bs u\|_{\X_p}\le R_{\bs f}$ and therefore $S(D_{R_{\bs f}})\subset D_{R_{\bs f}}$.

Given $\bs v_i\in D_{R_{\bs f}}$, $i=1,2,$ let $\bs u_i=S(\bs f,\gamma(\bs v_i)\varphi)$ and set $\mu=\frac{\gamma(\bs v_2)}{\gamma(\bs v_1)}$. We may assume $\mu>1$ without loss of generality. Setting $g=\gamma(\bs v_1)\varphi$, then $\mu g=\gamma(\bs v_2)\varphi$ and $S(\mu^{p-1}\bs f,\mu g)=\mu S(\bs f,g)$. Using \eqref{asymp2} with $\bs f_1=\bs f$ and $\bs f_2=\mu^{p-1}\bs f$, we have
\begin{align}\label{muuu}
\nonumber\|\bs u_1-\bs u_2\|_{\X_p}&\le \|S(\bs f,g)-S(\mu^{p-1}\bs f,\mu g)\|_{\X_p}+\|S(\mu^{p-1}\bs f,\mu g)-S(\bs f,\mu g)\|_{\X_p}\\
\nonumber&\le(\mu-1)\|\bs u_1\|_{\X_p}+(\mu^{p-1}-1)C_p\|\bs f\|_{L^{p'}\!\!(\Omega)^m}\\
&\le (\mu-1)p\, C_p\|\bs f\|_{L^{p'}\!\!(\Omega)^m},
\end{align}
since $\mu^{p-1}-1\le (p-1)(\mu-1)$, because $1<p\le2$, and $\|\bs u_1\|_{\X_p}\le C_p\|\bs f\|_{L^{p'}\!\!(\Omega)^m}$ from the estimate \eqref{asymp2} with $\bs f_1=\bs f$ and $\bs f_2=\bs 0$, where $C_p$ is given by \eqref{20} with $M=M(R_{\bs f})\|\varphi\|_{L^\infty(\Omega)}$.

Observing that, from the assumptions {\em i)} and {\em ii)},
\begin{equation*}
\mu-1=\frac{\gamma(\bs v_2)-\gamma(\bs v_1)}{\gamma(\bs v_1)}\le\tfrac{\Gamma(R_{\bs f})}{\eta(R_{\bs f})}\|\bs v_1-\bs v_2\|_{\X_p},
\end{equation*}
we get from \eqref{muuu}
$$\|S(\bs v_1)-S(\bs v_2)\|_{X_p}=\|\bs u_1-\bs u_2\|_{X_p}\le \tfrac{\Gamma(R_{\bs f})}{\eta(R_{\bs f})} p\, C_p\|\bs f\|_{L^{p'}\!\!(\Omega)^m}\|\bs v_1-\bs v_2\|_{X_p}.$$

Therefore the application $S$ is a contraction provided \eqref{Banach_c} holds and its fixed point $\bs u=S(\bs f, G[\bs u])$ solves uniquely \eqref{iqv}.
\qed\end{proof}

\begin{remark} The assumptions  {\em i)} and {\em ii)} are similar to the conditons in Appendix B of \cite{HintermullerRautenberg2012}, where the contractiveness of the solution application $S$ was obtained in an implicit form under the assumptions on the norm of $\bs f$ to be sufficiently small. Our expression \eqref{Banach_c} quantifies not only the size of the $L^{p'}$-norm of $\bs f$, but also  the constants of the functional $\gamma$, the $\varphi$ and the domain $\Omega$, through its measure and the size of its Poincar\'e constant.
\end{remark}

\subsection{Applications}

We present three examples of physical  applications.

\begin{example} {\bf A nonlinear Maxwell quasi-variational inequality} (see \cite{MirandaRodriguesSantos2009})\label{6}$ $

\noindent Consider a nonlinear electromagnetic field in equilibrium in a bounded simply connected domain $\Omega$ of $\R^3$.
We consider the stationary  Maxwell's equations 
\begin{equation*}
\bs j = \nabla\times\bs h,\quad \Rot\bs e=\bs f\quad \text{and}\quad \nabla\cdot\bs h=0
\quad\text{in}\ \Omega,
\end{equation*}
where $\bs j$, $\bs e$ and $\bs h$ denote, respectively, the current density, the electric and the magnetic fields.
For \mbox{type-II} superconductors we may assume constitutive laws of power type and an
extension of the Bean critical-state model, in which the current density cannot
exceed some given critical value $j\ge\nu>0$.  
When  $j$ may vary with the absolute value $|\bs h|$ of the
magnetic field (see Prigozhin, \cite{Pri1996-2} ) we obtain a quasi-variational inequality. Here we suppose
\begin{equation*}
\bs e=\begin{cases}
\delta|\Rot\bs h|^{p-2}\Rot\bs h&\mbox{ if }|\Rot\bs h|<j(|\bs h|),\vspace{1mm} \\
\big(\delta\,j^{p-2}+\lambda\big)\Rot\bs h&\mbox{ if }|\Rot\bs h|=j(|\bs h|),
\end{cases}
\end{equation*}
where $\delta\ge0$ is a  given constant
and $\lambda\geq0$ is an unknown Lagrange multiplier associated with the inequality constraint. The region $\big\{|\nabla\times\bs h|=j(|\bs h|)\big\}$ corresponds to the superconductivity region. We obtain the quasi-variational inequality \eqref{iqv} with $\X_p$ defined in \eqref{rot_normal0} or \eqref{rot_tangencial0}, depending whether we are considering a domain with perfectly conductive or perfectly permeable walls. 

The existence of solution is immediate by Theorem \ref{iqv_existence}.  1., if we assume $j:\X_p\rightarrow\R^+$ continuous, with $j\ge\nu>0$, for any $p>3$ and, for $1<p\le3$ if $j$ also has the growth condition of $F$ in  Example \ref{4}. above. Therefore, setting L$=\nabla\times$, for any $\bs f\in L^{p'}\!\!(\Omega)^3$ and any $\delta\ge0$, we have at least a solution to
\begin{equation*}
\left\{\begin{array}{l}
\bs h\in\K_{j(|\bs h|)}=\big\{\bs w\in\X_p:|\nabla\times\bs w|\le j(|\bs h|)\text{ in }\Omega\big\},\vspace{1mm}\\
\displaystyle\delta\int_\Omega|\nabla\times\bs h|^{p-2}\nabla\times\bs h\cdot\nabla\times(\bs w-\bs h)\ge\int_{\Omega}\bs f\cdot(\bs w-\bs h)\quad\forall\bs w\in \K_{j(|\bs h|)}.
\end{array}
\right.
\end{equation*}\qed
\end{example}

\vspace{1mm}

\begin{example} {\bf Thermo-elastic equilibrium of a locking material}	$ $
	
\noindent	Analogously to perfect plasticity, in 1957 Prager introduced the notion of an ideal locking material as a linear elastic solid for stresses below a certain threshold, which cannot be overpassed. When the threshold is attained,``there is locking in the sense that any further increase in stress will not cause any changes in strain" \cite{Prager1957}. Duvaut and Lions, in 1972 \cite{DuvautLions1972}, solved the general stationary problem in the framework of convex analysis. Here we consider a simplified situation for the displacement field $\bs u=\bs u(x)$ for $x\in\Omega\subset \R^d$, $d=1,2,3$, which linearized strain tensor $D\bs u=$L$\bs u$ is its symmetrized gradient. We shall consider $\X_2=H^1_0(\Omega)^d$ with norm $\|D\bs u\|_{L^2(\Omega)^{d^2}}$ and, for an elastic solid with Lam\'e constants $\mu>0$ and $\lambda\ge0$, we consider the quasi-variational inequality
\begin{equation}\label{lame}
\begin{cases}
\bs u\in\K_{b(\vartheta[\bs u])}=\big\{\bs w\in H^1_0(\Omega)^d:|D\bs w|\le b(\vartheta[\bs u])\text{ in }\Omega\big\}\vspace{1mm}\\
\displaystyle	\int_\Omega\Big(\mu D\bs u\cdot D(\bs w-\bs u)+\lambda\big(\nabla\cdot\bs u\big)\big(\nabla\cdot(\bs w-\bs u)\big)\Big)\vspace{1mm}\\
\hspace{5cm}\displaystyle\ge \int_\Omega\bs f\cdot(\bs w-\bs u)\quad\forall\bs w\in\K_{b(\vartheta[\bs u])}.
\end{cases}
\end{equation}
Here $b\in\C(\R)$, such that $b(\vartheta)\ge\nu>0$, is a continuous function of the temperature field $\vartheta=\vartheta[\bs u](x)$, supposed also in equilibrium under a thermal forcing depending on the deformation $D\bs u$. We suppose that $\vartheta[\bs u]$ solves
\begin{equation}\label{lame2}
-\Delta\vartheta=h(x,D\bs u(x))\  \text{in }\Omega,\quad \vartheta=0\ \text{ on }\partial\Omega,
\end{equation}
where $h:\Omega\times\R^{d^2}\rightarrow\R$ is a given Carath\'eodory function such that
\begin{equation}\label{lame3}
|h(x,D)|\le h_0(x)+C|D|^s,\quad\text{ for a.e. }x\in\Omega\ \text{ and }\ D\in\R^{d^2},
\end{equation}
for some function $h_0\in L^r(\Omega)$, with $r>\frac{d}2$ and $0<s<\frac{2}{r}$.
		
First, with $\bs w\equiv\bs 0$ in \eqref{lame}, we observe that any solution to \eqref{lame} satisfies the a priori bound $\|D\bs u\|_{L^2(\Omega){d^2}}\le\frac{k}\mu\|\bs f\|_{L^2(\Omega)^{d^2}}$, where $k$ is the constant of $\|\bs u\|_{L^2(\Omega)^{d}}\le k\|D\bs u\|_{L^2(\Omega)^{d^2}}$ from Korn's inequality.
		
Therefore, for each $\bs u\in H^1_0(\Omega)^d$, the unique solution $\vartheta\in H^1_0(\Omega)$ to \eqref{lame2} is in the H\"older space $\C^\alpha(\overline\Omega)$, for some $0<\alpha<1$, since $h=h(x,D\bs u(x))\in L^\frac{p}2(\Omega)$ by \eqref{lame3}, with the respective continuous dependence in $H^1_0(\Omega)\cap\C^\alpha(\overline\Omega)$ for the strong topologies, by De Giorgi-Stamppachia estimates (see, for instance, \cite[p. 170]{Rodrigues1987} and its references). By the a priori bound of $\bs u$ and the compactness of $\C^\alpha(\overline\Omega)\subset\C(\overline\Omega)$, if we define $G:\X_2\rightarrow\C(\overline\Omega)\cap L^\infty_\nu(\Omega)$ by $G[\bs u]=b(\vartheta[\bs u])$, we easily conclude that $G$ is a completely continuous operator and we can apply Theorem \ref{iqv_existence} to conclude that, for any $\bs f\in L^2(\Omega)^d$, $b\in\C(\R)$, $b\ge\nu>0$ and any $h$ satisfying \eqref{lame3}, there exists at least one solution $(\bs u,\vartheta)\in H^1_0(\Omega)^d\times\big(H^1_0(\Omega)\cap\C^\alpha(\overline\Omega)\big)$ to the coupled problem \eqref{lame}-\eqref{lame2}.\qed
\end{example}
		
\vspace{1mm}

\begin{example}  {\bf An ionization problem in electrostatics} (a new variant of  \cite{KunzeRodrigues2000})	$ $
	
\noindent  Let $\Omega$ be a bounded Lipschitz domain of $\R^d$, $d=2$ or $3$, being $\partial\Omega=\Gamma_0\cup\Gamma_1\cup\Gamma_\#$, with $\overline{\Gamma_0}\cap\overline{\Gamma_\#}\neq\emptyset$, both sets with positive $d-1$ Lebesgue measure. Denote by $\bs e$ the electric field, which we assume to be given by a potential $\bs e=-\nabla u$. We impose a potential difference between $\Gamma_0$ and $\Gamma_\#$ and that $\Gamma_1$ is insulated. So 
\begin{equation}\label{fronteira}
u=0\ \text{ on }\ \Gamma_0,\quad \bs j\cdot\bs n=0\ \text{ on }\ \Gamma_1\quad\text{ and }\quad  u=\ u_\#\ \text{ on }\ \Gamma_\#,
\end{equation}
with $\bs n$ being the outer unit normal vector to $\partial\Omega$. Here the trace $\ u_\#$ on $\Gamma_\#$ is an unknown constant to be found as part of the solution, by giving the total current $\tau$ across $\Gamma_\#$,
\begin{equation}\label{current}
\tau=\int_{\Gamma_\#}\bs j\cdot\bs n\in\R.
\end{equation}
We set L$=\nabla$, $V_2=H^1(\Omega)$ and, as in \cite{Rodrigues1996}, we define
\begin{equation}\label{space}
\X_2=H^1_\#=\big\{w\in H^1(\Omega):w=0\text{ on }\Gamma_0\text{ and }w=w_\#=\text{constant on }\Gamma_\# \big\},
\end{equation}
where the Poincar\'e inequality \eqref{poincare} holds, as well as the trace property for $w_\#=w_{|_{\Gamma_\#}}$, for some $c_\#>0$:
$$|w_\#|\le c_\#\|\nabla w\|_{L^2(\Omega)^d}\quad \forall w\in\X_2.$$

We assume, as in \cite[p.333]{DuvautLions1972} that
\begin{equation}\label{ili1}
\bs j=\begin{cases}
\sigma\bs e&\mbox{ if }|\bs e|<\gamma,\vspace{1mm} \\
(\sigma+\lambda)\bs e&\mbox{ if }|\bs e|=\gamma,
\end{cases}
\end{equation}     
where $\sigma$ is a positive constant, $\lambda\ge0$ is a Lagrange multiplier and $\gamma$ a positive ionization threshold. However, this is only an approximation of the true ionization law. In \cite{KunzeRodrigues2000}, it was proposed to let $\gamma$ vary locally with $|\bs e|^2$ in a neighbourhood of each point of the boundary, but here we shall consider instead that the ionization threshold depends on the difference of the potential on the opposite boundaries $\Gamma_0$ and $\Gamma_\#$, i.e. 
\begin{equation}
\label{ion}
\gamma=\gamma( u_\#)\ \text{ with }\ \gamma\in\C(\R)\ \text{ and }\ \gamma\ge\nu>0.
\end{equation}

Therefore we are led to search the electric potential $ u$ as the solution of the following quasi-variational inequality:
\begin{equation}\label{ionconvex}
u\in\K_{\gamma( u_\#)}=\big\{ w\in H^1_\#(\Omega):|\nabla w|\le\gamma( u_\#)\text{ in }\Omega\big\},
\end{equation}
\begin{equation}\label{ioniqv}
\sigma\int_\Omega\nabla u\cdot\nabla(w-u)\ge\int_{\Omega}f(w-u)-\tau(w_\#-u_\#)\quad\forall w\in \K_{\gamma(u_\#)},
\end{equation}
by incorporating the ionization law \eqref{ili1} with the conservation law of the electric charge $\nabla\cdot\bs j=f$ in $\Omega$ and the boundary conditions \eqref{fronteira} and \eqref{current} (see \cite{Rodrigues1996}, for details).

From \eqref{ioniqv} with $w=0$, we also have the a priori bound
\begin{equation}\label{apest}
\|\nabla u\|_{L^2(\Omega)^d}=\|u\|_{\X_2}\le\tfrac{c_2}{\sigma}\|f\|_{L^2(\Omega)}+\tfrac{c_\#}{\sigma}\equiv R_\#.
\end{equation}

Then, setting $G[u]=\gamma(u_\#)$ for $u\in\X_2=H^1_\#$, by the continuity of the trace on $\Gamma_\#$ and the assumption \eqref{ion}, we easily conclude that $G:\X_2\rightarrow[\nu,\gamma_\#]$ is a completely continuous operator, where $\gamma_\#=\displaystyle\max_{|r|\le c_\#\,R_\#}\gamma(r)$, with $R_\#$ from \eqref{apest}. Consequently, by Theorem \ref{iqv_existence}, there exists at least a solution to the ionization problem \eqref{ionconvex}-\eqref{ioniqv}, for any $f\in L^2(\Omega)$ and any $\tau\in\R$.

From \eqref{ioniqv}, if we denote by $w_1$ and $w_2$ the solutions of the variational inequality for $(f_1,\tau_1)$ and $(f_2,\tau_2)$ corresponding to the same convex $\K_g$ defined in \eqref{ionconvex}, we easily obtain the following version of Proposition \ref{dependstat}:
$$\|w_1-w_2\|_{H^1_\#(\Omega)}\le\tfrac{c_2}{\sigma}\|f_1-f_2\|_{L^2(\Omega)}+\tfrac{c_\#}{\sigma}|\tau_1-\tau_2|.$$

If, in addition, $\gamma\in \C^{0,1}(\R)$ and we set $\gamma'_\#=\displaystyle\sup_{|r|\le c_\#\,R_\#}|\gamma'(r)|$ we have
$$|\gamma(w_{1\#})-\gamma(w_{2\#})|\le\gamma_\#'|w_{1\#}-w_{2\#}|\le\gamma_\#'\,c_\#\|w_1-w_2\|_{H^1_\#(\Omega)}$$
and the argument of Theorem \ref{unistat} yields that the solution $u$ of \eqref{ionconvex}-\eqref{ioniqv} is unique provided that
$$2\gamma_\#'\,c_\#\big(\tfrac{c_2}{\sigma}\|f\|_{L^2(\Omega)}+\tfrac{c_\#}{\sigma}|\tau|\big)<\nu.$$\qed
\end{example}

\section{Evolutionary problems}\label{evolutiva}

\subsection{The variational inequality}\label{3.1}

For $T>0$ and $t\in(0,T)$, we set $Q_t=\Omega\times(0,t)$ and, for $\nu>0$, we define
$$L^\infty_\nu(Q_T)=\{w\in L^\infty(Q_T):w\ge\nu\}.$$ 
Given $g\in L^\infty_\nu(Q_T)$, for a.e. $t\in(0,T)$ we  set
$$\bs w\in\K_g\text{ iff }\bs w(t)\in\K_{g(t)}=\big\{\bs w\in\X_p:|\text{L}\bs w|\le g(t)\big\}.$$

We define, for $1<p<\infty$ and $p'=\frac{p}{p-1}$,
$$\V_p=L^p\big(0,T;\X_p\big),\quad\V_p'=L^{p'}\big(0,T;\X_p'\big),\quad\Y_p=\big\{\bs w\in\V_p:\partial_t\bs w\in\V_p'\big\}$$
and we assume that there exists an Hilbert space $\mathbb H$  such that
\begin{equation}\label{gelfand}
\mathbb H\subseteq L^2(\Omega)^m,\quad (\X_p,\mathbb H,\X_p')\text{ is a Gelfand triple,}\quad \X_p\hookrightarrow\mathbb H\text{ is compact}.
\end{equation}

As a consequence, by the embedding results of Sobolev-Bochner spaces (see, for instance \cite{Roubicek2013}), we have then
$$\Y_p\subset\C\big([0,T];\mathbb H)\subset\mathscr H\equiv L^p\big(0,T;\mathbb H\big)$$
and the embedding of $\Y_p\subset\mathscr H$ is also compact for $1<p<\infty$.

For $\delta\ge0$, given $\bs f:Q_T\rightarrow\R$ and $\bs u_0:\Omega\rightarrow\R$, $\bs u_0\in\K_{g(0)}$, we consider the weak formulation of the variational inequality, following \cite{LionsStamppachia1967},
\begin{equation}\label{iv_weak}
\left\{\begin{array}{l}
\bs u^\delta\in\K_{g},\vspace{1mm}\\
\displaystyle\int_0^T\langle\partial_t\bs w,\bs w-\bs u^\delta\rangle_p+\delta\int_{Q_T} \L_p\bs u^\delta\cdot \text{L}(\bs w-\bs u^\delta)\ge\displaystyle\int_{Q_T}\bs f\cdot(\bs w-\bs u^\delta)\vspace{1mm}\\
\hfill{\hspace{2cm}\displaystyle-\frac12\int_\Omega|\bs w(0)-\bs u_0|^2},\quad\forall \bs w\in \K_{g}\cap\Y_p
\end{array}
\right.
\end{equation}
and we observe that the solution $\bs u^\delta\in\V_p$ is not required to have the time derivative $\partial_t\bs u^\delta$ in the dual space $\V_p'$ and satisfies the initial condition in a very weak sense. In \eqref{iv_weak}, $\langle\,\cdot\,,\,\cdot\,\rangle_p$ denotes the duality pairing between $\X_p'$ and $\X_p$, which reduces to the inner problem in $L^2(\Omega)^m$ if both functions belong to this space.

When $\partial_t\bs u^\delta\in L^2\big(0,T;L^2(\Omega)^m\big)$ (or more generally when $\bs u^\delta\in\Y_p$), the strong formulation reads
\begin{equation}\label{ive}
\left\{\begin{array}{l}
\bs u^\delta(t)\in\K_{g(t)},\  \,t\in[0,T],\ \,\bs u(0)=\bs u_0,\vspace{1mm}\\
\displaystyle\int_{\Omega}\partial_t\bs u^\delta(t)\cdot(\bs w-\bs u^\delta(t))+\delta\int_{\Omega} \L_p\bs u^\delta(t)\cdot \text{L}(\bs w-\bs u^\delta(t))\vspace{1mm}\\
\hspace{2cm}\displaystyle\ge\int_{\Omega}\bs f(t)\cdot(\bs w-\bs u^\delta(t)),
\displaystyle\quad\forall\,\bs w\in \K_{g(t)}\ \text{ for a.e. }t\in(0,T).
\end{array}
\right.
\end{equation}

Integrating \eqref{ive} in $t\in(0,T)$ with $\bs w\in\K_g\cap\Y_p\subset\C\big([0,T];L^2(\Omega)^m\big)$ and using
\begin{multline*}
\int_0^t\langle\partial_t\bs u^\delta-\partial_t\bs w,\bs w-\bs u^\delta\rangle_p
=\frac12\int_\Omega|\bs w(0)-\bs u_0|^2-\frac12\int_\Omega|\bs w(t)-\bs u^\delta(t)|^2\\
\le\frac12\int_\Omega|\bs w(0)-\bs u_0|^2
\end{multline*}
we immediately conclude that a strong solution is also a weak solution, i.e., it satisfies \eqref{iv_weak}. Reciprocally, if $\bs u^\delta\in\K_g$ with $\partial_t\bs u^\delta\in L^2(Q_T)^m$ (or if $\bs u^\delta\in\Y_p$) is a weak solution with $\bs u^\delta(0)=\bs u_0$, replacing in \eqref{iv_weak} $\bs w$ by $\bs u^\delta+\theta(\bs z-\bs u^\delta)$ for $\theta\in(0,1]$ and $\bs z\in\K_g\cap\Y_p$, and letting $\theta\rightarrow0$, we conclude that $\bs u^\delta$ also satisfies
$$\int_{Q_T}\partial_t\bs u^\delta\cdot(\bs z-\bs u^\delta)+\delta\int_{Q_T}\L_p\bs u^\delta\cdot\text{L}(\bs z-\bs u^\delta)\ge\int_{Q_T}\bs f\cdot(\bs z-\bs u^\delta)$$
 and, by approximation, when $g\in\C\big([0,T];L^\infty_\nu(\Omega)\big)$ (see \cite[Lemma 5.2]{MirandaRodriguesSantos2018}), also for all $\bs z\in \K_g$.

For any $\bs w\in\K_{g(t)}$, for fixed $t\in(0,T)$ and arbitrary $s$, $0<s<t<T-s$, we can use as test function in \eqref{ive} $\bs z\in\K_g$ such that $\bs z(\tau)=0$ if $\tau\not\in(t-s,t+s)$ and $\bs z(\tau)=\frac\nu{\nu+\eps_s}\bs w$ if $\tau\in(t-s,t+s)$, with $\displaystyle\eps_s=\sup_{t-s<\tau<t+s}\|g(t)-g(\tau)\|_{L^\infty(\Omega)}$. Hence, dividing by $2s$ and letting $s\rightarrow0$, we can conclude the equivalence between \eqref{ive} and \eqref{iv_weak}.

We have the following existence and uniqueness result whose proof, under more general assumptions for monotone operators,  can be found in \cite{MirandaRodriguesSantos2018}.

\begin{theorem}\label{teoIQV}
Suppose that  $\delta\ge0$ and \eqref{L},  \eqref{Xp}, \eqref{norm} and \eqref{gelfand} are satisfied. Assume that
\begin{equation}\label{IVweakAssumptions}
\bs f\in L^2(Q_T)^m,\quad g\in\C\big([0,T];L^\infty_\nu(\Omega)\big),\quad  \bs u_0\in\K_{g(0)}.
\end{equation}
Then, for any $\delta\ge0$,  the variational inequality \eqref{iv_weak} has a unique weak solution 
$$\bs u^\delta\in\V_p\cap\C\big([0,T];L^2(\Omega)^m\big).$$
	
If, in addition,
\begin{equation}\label{gforte}
g\in W^{1,\infty}\big(0,T;L^\infty(\Omega)\big),\quad g\ge\nu>0
\end{equation}
then the variational inequality \eqref{ive} has a unique strong solution 
$$\bs u^\delta\in\V_p\cap H^1\big(0,T;L^2(\Omega)^m\big).$$
\qed\end{theorem}

\begin{remark} For the scalar case L$=\nabla$, with $p=2$, a previous result for strong solutions was obtained in \cite{Santos2002} with $g\in\C(\overline Q_T)\cap W^{1,\infty}\big(0,T;L^\infty(\Omega)\big)$, $g\ge\nu>0$ for the coercive case $\delta>0$. More recently, a similar result was obtained with the time-dependent subdifferential operator techniques by Kenmochi in \cite{Ken2013}, also for $\delta>0$ and for the scalar case L=$\nabla$, getting weak solutions for $1<p<\infty$ with $g\in\C(\overline Q_T)$ and strong solutions with $g\in\C(\overline Q_T)\cap H^1\big(0,T;\C(\overline\Omega)\big)$.
\end{remark}

The next theorem gives a quantitative result on the continuous dependence on the data, which essentially establishes the Lipschitz continuity of the solutions with respect to $\bs f$ and $\bs u_0$ and the H\"older continuity (up to $\frac12$ only) with respect to the threshold $g$. This estimate in $\V_p$ was obtained first in \cite{Santos2002} with L$=\nabla$ and $p=2$ and developed later in several other works, including \cite{MirandaRodriguesSantos2012}, \cite{Ken2013} and \cite{MirandaRodriguesSantos2018}. Here we give an explicit dependence of the constants with respect to the data.

\begin{theorem}\label{cdde} 	Suppose that  $\delta\ge0$ and \eqref{L},  \eqref{Xp}, \eqref{norm} and \eqref{gelfand} are satisfied. Let $i=1,2,$ and suppose that $\bs f_i\in L^2(\Omega)^m$, $g_i\in\C\big([0,T];L^\infty_\nu(\Omega)\big)$ and $\bs u_{0i}\in\K_{g_i(0)}$. If $\bs u^{\delta}_i$ are the solutions of the variational inequality \eqref{iv_weak}  with data $(\bs f_i,\bs u_{0i},g_i)$ then there exists a constant  $B$, which depends only in a monotone increasing way on $T$, $\|\bs u_{ 0i}\|_{L^2(\Omega)^m}^2$ and $\|\bs f_i\|_{L^2(\Omega)^m}^2$, such that
\begin{multline}\label{stabilityef}
\|\bs u^{\delta}_1-\bs u^{\delta}_2\|_{L^\infty(0,T;L^2(\Omega)^m)}^2\le(1+Te^T) \Big(\|\bs f_1-\bs f_2\|_{L^2(Q_T)^m}^2\\
+\|\bs u_{ 01}-\bs u_{ 02}\|_{L^2(\Omega)^m}^2+\tfrac{B}{\nu}\|g_1-g_2\|_{L^\infty(Q_T)}\Big).
\end{multline} \vspace{-5mm}
	
Besides, if $\delta>0$, 
\begin{multline}\label{stabilityeff}
\|\bs u^{\delta}_1-\bs u^{\delta}_2\|_{\V_p}^{p\vee2}\le \tfrac{a_p}{\delta}\Big(\|\bs f_1-\bs f_2\|_{L^2(Q_T)^m}^2\\
+\|\bs u_{ 01}-\bs u_{ 02}\|_{L^2(\Omega)^m}^2+\tfrac{B}\nu\|g_1-g_2\|_{L^\infty(Q_T)}\Big),
\end{multline}
where 
\begin{equation}\label{ap}
a_p=\tfrac{(1+T+T^2e^T)}{2\,d_p}\,\big(c_g\,\text{\footnotesize $|Q_T|^\frac1p$}\big)^{(2-p)^+}\quad 1<p<\infty,
\end{equation}
being $d_p$ given by \eqref{monotone} and $c_g=\|g_1\|_{L^\infty(Q_T)}+
\|g_2\|_{L^\infty(Q_T)}$.
\end{theorem}
\begin{proof} We prove first the result for strong solutions, approximating the function  $g_i$ in $\C\big([0,T];L^\infty_\nu(\Omega)\big)$ by a sequence $\{g_i^n\}_n$ belonging to $W^{1,\infty}\big(0,T;L^\infty(\Omega)\big)$. 
	
Given two strong solutions $\bs u_i^\delta$, $i=1,2$, setting $\beta=\|g_1-g_2\|_{L^\infty(Q_T)}$,  denoting $\overline{\bs u}=\bs u^\delta_1-\bs u^\delta_2$, $\overline{\bs u}_0=\overline{\bs u}_{01}-\overline{\bs u}_{02}$, $\overline{\bs f}=\bs f_1-\bs f_2$, $\overline g=g_1-g_2$ and $\alpha=\frac\nu{\nu+\beta}$ and using the test functions $\bs w_{i_j}=\frac{\nu\bs u^\delta_i}{\nu+\beta}\in\K_j$, for $i,j=1,2$, $i\neq j$, we obtain the inequality
\begin{equation}\label{bar}
\int_\Omega\partial_t\overline{\bs u}(t)\cdot\overline{\bs u}(t)+\delta\int_{\Omega}\big(\L_p{\bs u}^\delta_1(t)-\L_p{\bs u}^\delta_2(t)\big)\cdot\text{L}\overline{\bs u}(t)\le\int_{\Omega}\overline{\bs f}(t)\cdot\overline{\bs u}(t)+\Theta(t),
\end{equation}
where
\begin{multline}
\Theta(t)=(\alpha-1)\int_{\Omega}\big( \partial_t(\bs u^\delta_1\cdot\bs u^\delta_2)\\
+\delta\L_p\bs u^\delta_1\cdot\text{L}\bs u^\delta_2+\delta\L_p\bs u^\delta_2\cdot\text{L}\bs u^\delta_1+\bs f_1\cdot\bs u^\delta_2+\bs f_2\cdot\bs u^\delta_1\big)(t)
\end{multline}
and, because  $1-\alpha=\frac{\beta}{\beta+\nu}\le\frac1\nu\|g_1-g_2\|_{L^\infty(Q_T)}$, then for any $t\in(0,T)$
\begin{equation}\label{theta}
\int_0^t\Theta\,d\tau\le \tfrac{B}{2\nu}\|g_1-g_2\|_{L^\infty(Q_T)},
\end{equation}
where the constant $B$ depends on   $\|\bs f_i\|_{L^2(Q_T)^m}$ and $\|\bs u_{0i}\|_{L^2(\Omega)}$.
From \eqref{bar}, we have
$$\int_{\Omega}|\overline{\bs u}(t)|^2\le \int_0^t\int_{\Omega}|\overline{\bs u}|^2+\int_{\Omega}|\overline{\bs u}_0|^2+\int_{Q_T}|\overline{\bs f}|^2+2\int_0^T\Theta,$$
proving \eqref{stabilityef} by applying the integral Gronwall inequality.
		
If $\delta>0$ and $p\ge2$, using the monotonicity of $\L_p$,  then
\begin{equation*}
\int_{\Omega}|\overline{\bs u}(t)|^2+2\,\delta\,d_p\int_{Q_t}|\text{L}\overline{\bs u}|^p_{L^p(Q_t)^\ell}\le\int_{Q_t}|\overline{\bs f}|^2+\int_{Q_t}|\overline{\bs u}|^2+\int_\Omega|\overline{\bs u}_0|^2+2\int_0^t\Theta(\tau)
\end{equation*}
and,  by the estimates \eqref{stabilityef} and \eqref{theta}, by integrating in $t$ we easily obtain \eqref{stabilityeff}.
	
For $\delta>0$ and $1<p<2$ set $c_g=\|g_1\|_{ L^\infty(Q_T)}+\|g_2\|_{ L^\infty(Q_T)}$. So, using the monotonicity \eqref{monotone} of $\L_p$ and the H\"older inverse inequality,
\begin{multline*}
\int_{\Omega}|\overline{\bs u}(t)|^2+2\,\delta\, d_p\,c_g^{p-2}|Q_T|^\frac{p-2}{p}\Big(\int_{Q_T}|\text{L}\overline{\bs u}(t)|^p\Big)^\frac2p\\
\le \int_{Q_t}|\overline{\bs f}|^2+\int_{Q_t}|\overline{\bs u}|^2+\int_\Omega|\overline{\bs u}_0|^2+\tfrac{B}{\nu}\|g_1-g_2\|_{L^\infty(Q_T)}
\end{multline*}
and using the estimate \eqref{stabilityef} to control $\|\overline{\bs u}\|^2_{L^2(Q_T)^m}$ as above,
we conclude the proof for strong solutions.
	
To prove the results for weak solutions, it is enough to recall that they can be approximated by strong solutions in $\C([0,T];L^2(\Omega)^m)\cap\V_p$.
\qed\end{proof}

Using the same proof for the case $\nabla\times$ of \cite{MirandaRodriguesSantos2012}, which was a development of the scalar case with $p=2$ of \cite{Santos2002}, we can prove the asymptotic behaviour of the strong solution of the variational inequality when $t\rightarrow\infty$. Consider the stationary variational inequality \eqref{iv} with data $\bs f_\infty$ and $g_\infty$ and  denoting its solution by $\bs u_\infty$, we have the following result.
\begin{theorem} \label{asympttt}Suppose that  the assumptions \eqref{L},  \eqref{Xp},  \eqref{norm} and \eqref{gelfand} are satisfied and
$$\bs f\in L^\infty\big(0,\infty;L^2(\Omega)^m\big),\qquad g\in W^{1,\infty}\big(0,\infty;L^\infty(\Omega)\big),\ g\ge\nu>0,$$
$$\bs f_\infty\in L^2(\Omega)^m,\qquad g_\infty\in L^\infty_\nu(\Omega),$$
$$\int_{\frac{t}2}^t\xi^{p'}(\tau)d\tau\underset{t\rightarrow\infty}{\longrightarrow}0,\quad \text{ if }p>2\qquad\text{and}\qquad\int_t^{t+1}\xi^2(\tau)d\tau\underset{t\rightarrow\infty}{\longrightarrow}0\quad \text{ if }1< p\le2,$$	where
\begin{equation}
\xi(t)=\|\bs f(t)-\bs f_\infty\|_{L^2(\Omega)^m}.
\end{equation}
	
Assume, in addition, that there exist $D$ and $\gamma$ positive such that
\begin{equation}
\|g(t)-g_\infty\|_{L^\infty(\Omega)}\le\frac{D}{t^\gamma},\qquad\text{where}\ \  \gamma>\begin{cases}\frac32&\text{ if }p> 2\\
\frac12&\text{ if }1<p\le2.\end{cases}
\end{equation}
	
Then, for $\delta>0$ and $\bs u^\delta$ the solution of the variational inequality \eqref{ive}, with $t\in[0,\infty)$,
$$\|\bs u^\delta(t)-\bs u_\infty^\delta\|_{L^2(\Omega)^m}\underset{t\rightarrow\infty}{\longrightarrow}0.$$
\qed
\end{theorem}

In the special case of \eqref{ive} with $\delta=0$ and $g(t)=g$ for al $t\ge T^*$, observing that we can apply a result of Br\'ezis \cite[Theorem 3.11]{Brezis1973} to extend the Theorem 3.4 of  \cite{DumontIgbida2009}), in which $g\equiv1$, and obtain the following asymptotic behaviour of the solution $\bs u(t)\in\K_g$ with $\bs u(0)=\bs u_0$ of
\begin{equation}
\label{above}
\int_{\Omega}\partial_t\bs u(t)\cdot(\bs v-\bs u(t))\ge\int_\Omega\bs f(t)\cdot(\bs v-\bs u(t))\quad\forall\bs v\in\K_{g},
\end{equation}
which corresponds, in the scalar case, to the sandpile problem with space variable slope.
	
\begin{theorem} \label{asympt22}Suppose that \eqref{L}, \eqref{Xp}, \eqref{norm} and \eqref{gelfand} are satisfied, $\bs f\in L^1_{loc}\big(0,\infty;L^2(\Omega)^m)$, $g\in L_\nu^\infty(\Omega)$, $\bs u_0\in\K_{g}$ and let $\bs u$ be the solution of the variational inequality \eqref{above}. If there exists a function $\bs f_\infty$ such that $\bs f-\bs f_\infty\in L^1\big(0,\infty;L^2(\Omega)^m\big)$ then 
$$\bs u(t)\underset{t\rightarrow\infty}{\longrightarrow}\bs u_\infty\quad\text{ in }L^2(\Omega)^m,$$
where $\bs u_\infty$ solves the variational inequality \eqref{ivdeg} with $\bs f_\infty$.\qed
\end{theorem}

\subsection{Equivalent formulations when L=$\nabla$}

In this section, we summarize the main results of \cite{Santos2002}, assuming $\partial\Omega$ is of class $\C^2$, $p=2$ and $L=\nabla$ and considering the strong  variational inequality \eqref{iv} in this special case, 
\begin{equation}\label{ivgrad}
\left\{\begin{array}{l}
u(t)\in\K_{g(t)},\    u(0)=  u_0,\vspace{1mm}\\
\displaystyle\int_{\Omega}\partial_t   u(t)\cdot(   v-   u(t))+\int_{\Omega} \nabla   u(t)\cdot \nabla(   v-   u(t))\ge\int_{\Omega}   f(t)\cdot(   v-   u(t)),
\vspace{1mm}\\
\hfill{\displaystyle\forall\,   v\in \K_{g(t)}\ \text{ for a.e. }t\in(0,T),}
\end{array}
\right.
\end{equation}
where 
$$\K_{g(t)}=\big\{v\in H^1_0(\Omega):|\nabla v|\le g(t)\big\}.$$
As in the stationary case, we can consider three related problems.
The first one is the Lagrange multiplier problem
\begin{equation*}
\int_{Q_T}\partial_t u\varphi+\langle\lambda\nabla u,\nabla\varphi\rangle_{(L^\infty(Q_T)'\times L^\infty(Q_T)}= \int_{Q_T}  f\varphi,\quad\forall\varphi\in L^\infty\big(0,T;W^{1,\infty}_0(\Omega)\big),
\end{equation*}\vspace{-4mm}
\begin{equation}
\label{weaklme}
\lambda\ge 1,\quad(\lambda-1)(|\nabla u|-g)=0\quad \text{ in }\big(L^\infty(Q_T)\big)',
\end{equation}
\begin{equation*}
u(0)=u_0,\ \text{ a.e. in }\Omega\quad |\nabla u|\le g\ \text{ a.e. in }Q_T.
\end{equation*}
which is equivalent to the variational inequality \eqref{ivgrad}. This was first proved in \cite{Santos1991} in the case $g\equiv1$, where it was shown the existence of $\lambda\in L^\infty(Q_T)$ satisfying \eqref{weaklme}, in the case of a compatible and smooth nonhomogeneous boundary condition for $u$. When $u_{|_{\partial\Omega\times(0,T)}}$ is independent of $x\in\partial\Omega$ then, by Theorem 3.11 of \cite{Santos1991}, $\lambda$ is unique. In this framework, it was also shown in \cite{Santos1991} that the solution $u\in L^p\big(0,T;W^{2,p}_{loc}(\Omega)\big)\cap \C^{1+\alpha,\alpha/2}(Q_T)$ for all $1\le p<\infty$ and $0\le\alpha<1$.

Secondly, we define two obstacles as in \eqref{sup} and \eqref{inf} using the pseudometric $d_{g(t)}$ introduced in \eqref{dg},
\begin{eqnarray}\label{ov}
\overline{\varphi}(x,t)=d_{g(t)}(x,\partial\Omega)=\bigvee\{w(x):w\in\K_{g(t)}\}
\end{eqnarray}
and
\begin{eqnarray}\label{uv}
\underline{\varphi}(x,t)=d_{g(t)}(x,\partial\Omega)=\bigwedge\{w(x):w\in\K_{g(t)}\}, 
\end{eqnarray}
where the variable $\K_{\underline\varphi(t)}^{\overline\varphi(t)}$ is defined by \eqref{iv2ob} for each $t\in[0,T]$
and we consider the double obstacle variational inequality
\begin{equation}\label{iv2obe}
\left\{\begin{array}{l}
u(t)\in\K_{\underline{\varphi}(t)}^{\overline{\varphi}(t)},\    u(0)=  u_0,\vspace{1mm}\\
\displaystyle\int_{\Omega}\partial_t   u(t)\cdot(   v-   u(t))+\int_{\Omega} \nabla   u(t)\cdot \nabla(   v-   u(t))\ge\int_{\Omega}   f(t)\cdot(   v-   u(t)),
\vspace{1mm}\\
\hfill{\displaystyle\forall\,   v\in \K_{\underline{\varphi}(t)}^{\overline{\varphi}(t)}\ \text{ for a.e. }t\in(0,T),}
\end{array}
\right.
\end{equation}

The third and last problem is the following complementary problem
\begin{align}\label{peve}
\nonumber&(\partial_t u-\Delta u-f)\vee(|\nabla u|-g)= 0\quad \mbox{ in }Q_T,\\
&u(0)=u_0\quad\mbox{ in }\Omega,\quad u=0\quad\text{ on }\partial\Omega\times(0,T).
\end{align}
In \cite{Zhu1992}, Zhu studied a more general problem in unbounded domains, for  large times, with a zero condition at a fixed instant $T$, motivated by stochastic control.

These different formulations of gradient constraint problems are not always equivalent and were studied in \cite{Santos2002}, where sufficient conditions were given for the equivalence of each one with \eqref{ivgrad}. 

Assume that
\begin{align}\label{as_g_varios}
\nonumber& g\in W^{1,\infty}\big(0,T; L^\infty(\Omega)\big)\cap L^\infty\big(0,T;\C^2(\overline\Omega)\big), \quad g\ge\nu>0,\\
&|\nabla w_0|\le g(0),\quad f\in L^\infty(Q_T).
\end{align}

The first result holds with an additional assumption on the gradient constraint $g$, which is, of course, satisfied in the case  of $g\equiv$constant$>0$, by combining Theorem 3.9 of \cite{Santos2002} and Theorem 3.11 of \cite{Santos1991}.

\begin{theorem} 
Under the assumptions  \eqref{as_g_varios}, with $f\in L^\infty(0,T)$ and
\begin{equation}\label{supcalcal}
\partial_t(g^2)\ge0,\qquad -\Delta(g^2)\ge0,
\end{equation}
problem  \eqref{weaklme} has a solution 
$(\lambda,u)\in L^\infty(Q_T)\times L^\infty\big(0,T;W^{1,\infty}_0(\Omega)\cap H^2_{loc}(\Omega)\big).$	Besides, $u$ is the unique solution of \eqref{ivgrad} and if $g$ is constant then $\lambda$ is unique.\qed
\end{theorem}

The equivalence with the double obstacle problem holds with a slightly weaker assumption on $g$.

\begin{theorem} Assuming \eqref{as_g_varios}, problem \eqref{iv2obe} has a unique solution. If $f\in L^\infty(0,T)$ and
\begin{equation}\label{supcal}
\partial_t(g^2)-\Delta(g^2)\ge0,
\end{equation}
then  problem \eqref{iv2obe} is equivalent to problem \eqref{ivgrad}.\qed
\end{theorem}

Finally, the sufficient conditions for the equivalence of the complementary problem \eqref{peve} and the gradient constraint scalar problem \eqref{ivgrad} require stronger assumptions on the data.

\begin{theorem} Suppose that $f\in W^{1,\infty}\big(0,T;L^\infty(\Omega)\big)$, $w_0\in H^1_0(\Omega)$,  and  
\begin{align*}&\Delta u_0\in L^\infty(\Omega),\ -\Delta u_0\le f\text{ a.e. in }Q_T,\\
& g\in W^{1,\infty}\big(0,T;L^\infty(\Omega)\big)\ g\ge\nu>0\text{ and }\ \partial_t(g^2)\le0.
 \end{align*}
Then problem \eqref{peve} has a unique solution. If, in addition, $g=g(x)$  and $\Delta g^2\le0$  then this problem is equivalent to problem \eqref{ivgrad}.\qed
\end{theorem}

The counterexample given at the end Section \ref{2.3}, concerning the non-equivalence among these problems, can be generalized easily for the evolutionary case, as we have stabilization in time to the stationary solution (see \cite{Santos2002}).

\subsection{The scalar quasi-variational inequality with gradient constraint}\label{ivgde}

In \cite{RodriguesSantos2012}, Rodrigues and Santos proved existence of solution for a quasi-variational inequality with gradient constraint for first order quasilinear equations ($\delta=0$), extending the previous results for parabolic equations of \cite{RodriguesSantos2000}.

For $\bs\Phi=\bs\Phi(x,t,u):\overline Q_T\times\R\rightarrow\R^d$, $F=F(x,t,u):\overline Q_T\times\R\rightarrow\R$ assume that
\begin{equation}\label{iqv0_ass}
\bs\Phi\in W^{2,\infty}\big(Q_T\times(-R,R)\big)^d,\qquad F\in W^{1,\infty}\big(Q_T\times(-R,R)\big).
\end{equation}
In addition,
$\nabla\cdot\bs\Phi$ and $F$ satisfy the growth condition in the variable $u$
\begin{equation}\label{growth_ass}
|\big(\nabla\cdot\bs\Phi\big)(x,t,u)+F(x,t,u)|\le c_1|u|+c_2,
\end{equation}
uniformly in $(x,t)$, for all $u\in\R$ and a.e. $(x,t)$, being $c_1$ and $c_2$ positive constants. The gradient constraint $G=G(x,u):\Omega\times\R\rightarrow\R$ is bounded in $x$ and continuous in $ u$ and the initial condition $u_0:\Omega\rightarrow\R$  are such that
\begin{equation}\label{gxu}
G\in\C\big(\R;L^\infty_\nu(\Omega)\big),\qquad u_0\in\K_{G(u_0)}\cap\C(\overline\Omega),\qquad\delta\Delta_p u_0\in M(\Omega),
\end{equation}
being
$$\K_{G(u(t))}=\big\{w\in H^1_0(\Omega):|\nabla w|\le G(u(t))\big\}$$
and $M(\Omega)$ denotes the space of bounded measures in $\Omega$.
\begin{theorem}\label{arma} Assuming \eqref{iqv0_ass}, \eqref{growth_ass} and \eqref{gxu}, for each $\delta\ge0$ and any $1<p<\infty$,
	the  quasi-variational inequality
\begin{equation}
\label{iqvpisaarma}
\begin{cases}
u (t)\in\K_{G(u (t))} \text{  for a.e. }t\in(0,T), u(0)=u_0,\vspace{0.5mm}\\
\displaystyle\langle\partial_tu (t),w-u \rangle_{M(\Omega)\times\C(\overline \Omega)}+\int_\Omega\big(\delta\nabla_pu(t)+\bs\Phi(u (t))\big)\cdot\nabla(w(t)-u(t))\vspace{0.5mm}\\
\displaystyle\qquad\hfill{\ge\int_\Omega
F(u(t))(w-u (t))\qquad\forall w\in\K_{G(u (t))}, \text{ for a.e. }t\in(0,T)},
\end{cases}
\end{equation}
has a solution $u \in L^\infty\big(0,T;W^{1,\infty}_0(\Omega)\big)\cap\C(\overline Q_T)$ such that $\partial_tu \in L^\infty\big(0,T;M(\Omega)\big)$.\qed
\end{theorem}

Although this result was proved in \cite{RodriguesSantos2012} for $\delta=0$  and only in the case $p=2$ for $\delta>0$, it can be proved for $p\neq 2$ exactly in the same way as in the previous framework of \cite{RodriguesSantos2000}, which corresponds to \eqref{iqvpisaarma} when $\bs\Phi\equiv\bs0$,  with $G(x,u)=G(u)$ and $F(x,t,u)=f(x,t)$, with only $f\in L^\infty(Q_T)$ and $\partial_tf\in M(\Omega)$.

We may consider the corresponding stationary quasi-variational inequality for $ u_\infty\in\K_{G[ u_\infty]}$, such that
\begin{equation}\label{estasymp}
\int_\Omega\big(\delta\nabla_p u_\infty+\bs\Phi_\infty( u_\infty)\big)\cdot\nabla(w- u_\infty)\ge\int_\Omega F_\infty( u_\infty)(w-u_\infty)\qquad\forall w\in\K_{G[ u_\infty]}
\end{equation}
for given functions $F_\infty=F_\infty(x,u):\overline{\Omega}\times\R$ and $\bs\Phi_\infty=\bs\Phi_\infty(x,u):\Omega\times\R\rightarrow\R^d$, continuous in $u$ and bounded in $x$ for all $|u|\le R$. In order to extend the asymptotic stabilization in time (for subsequences $t_n\rightarrow\infty$) obtained in \cite{RodriguesSantos2000} and \cite{RodriguesSantos2012}, we shall assume that \eqref{iqv0_ass} holds for $T=\infty$,
\begin{equation}
\label{hipinfty}
\bs\Phi(t)=\bs\Phi_\infty,\qquad\text{and}\qquad\partial_u F\le-\mu<0\ \text{ for all }t>0,
\end{equation}
and
\begin{equation}
\label{hipinfty2}
0<\nu\le G(x,u)\le L,\ \text{ for a.e. }x\in\Omega\quad\text{and all }  u\in\R,
\end{equation}
or there exists $M>0$ such that, for all $R\ge M$
\begin{equation}
\label{hipinfty3}
\nabla\cdot\bs\Phi(x,R)+F(x,t,R)\le0,\qquad \nabla\cdot\bs\Phi(x,-R)+F(x,t,-R)\ge0.
\end{equation}

Setting $\xi_R(t)=\displaystyle\int_\Omega\sup_{|u|\le R}\big|\partial_t F(x,t,u)\big|dx$ and supposing that, for $R\ge R_0$ and some constant $C_R>0$ we have
\begin{equation}
\label{asympt}
\sup_{0<t<\infty}\int_t^{t+1}\xi_R(\tau)d\tau\le C_R,\qquad{and}\qquad\int_t^{t+1}\xi_R(\tau)d\tau\underset{t\rightarrow\infty}{\longrightarrow}0,
\end{equation}
and
\begin{equation}
\label{asympt2}
F(x,t,u)\underset{t\rightarrow\infty}{\longrightarrow}F_\infty(x,u)\ \text{ for all }|u|\le R\ \text{ and a.e. }x\in\Omega,
\end{equation}
we may prove the following result
\begin{theorem}\label{tasymp}
For any $\delta\ge0$ and $1<p<\infty$, under the assumptions \eqref{iqv0_ass}-\eqref{gxu}  and \eqref{hipinfty}-\eqref{asympt2}, problem \eqref{estasymp} has a solution $u_\infty\in\K_{G[u_\infty]}$ which is the weak-$*$ limit in $W^{1,\infty}_0(\Omega)$ and strong limit in $\C^\alpha(\overline{\Omega})$, $0\le\alpha<1$, for some $t_n\rightarrow\infty$ of a sequence $\{u(t_n)\}_n$ with $u$ being a global solution of \eqref{iqvpisaarma}.\qed
\end{theorem} 

We observe that the degenerate case $\delta=0$ corresponds to a nonlinear conservation law for which we could also consider formally the Lagrange multiplier problem of finding $\lambda=\lambda(x,t)$ associated with the constraint $|\nabla u|\le G[u]$ and such that
$$\lambda\ge0,\qquad\lambda(G[u]-|\nabla u|)=0$$
and
$$\partial_tu-\nabla\cdot\Phi(u)-\nabla\cdot(\lambda u)=F(u),\ \text{ in }Q_T,$$
which is an open problem. However, in contrast with conservation laws without the gradient restriction, this problem has no spatial shock fronts nor boundary layers for the vanishing viscosity limit, since Dirichlet data may be prescribed for $u$ on the whole boundary $\partial\Omega$.

It is clear that both results of Theorems~\ref{arma} and \ref{tasymp} apply, in particular, to the linear transport equation
$$\partial_tu+\bs b\cdot\nabla u+cu=f\ $$
for a given vector field $\bs b\in W^{2,\infty}(Q_T)^d$ and given functions $c=c(x,t)$ and $f=f(x,t)$ in $W^{1,\infty}(Q_T)$, corresponding to set
\begin{equation}\label{phib}
\bs\Phi(u)=-\bs b u\quad\text{and}\quad F(u)=f-(c+\nabla\cdot\bs b)u.
\end{equation}

Nevertheless, in this case, if we analyse the a priori estimates for the approximating problem in the proof of Theorem~\ref{arma} as in Section 3.1 of \cite{RodriguesSantos2012}, the assumptions on the coefficients of the linear transport operator can be significantly weakened and it is possible to prove the following result
\begin{corollary} \label{coro} If $\bs \Phi$ and $F$ are given by \eqref{phib} with $\bs b\in L^\infty(Q_T)^d $, $\nabla\cdot\bs b$, $c$, $f\in L^\infty(Q_T)$, $\partial_t\bs b\in L^r(Q_T)^d$, $r>1$ and $\partial_tc$, $\partial_tf\in L^1(Q_T)$, assuming \eqref{gxu} for $\delta\ge0$ and $1<p<\infty$, the quasi-variational inequality \eqref{iqvpisaarma} with linear lower order terms has a strong solution $\bs u\in L^\infty\big(0,T;W^{1,\infty}_0(\Omega)\big)\cap\C(\overline Q_T)$ such that $\partial_t u\in L^1\big(0,T;M(\Omega)\big)$.
\qed
\end{corollary}

However, in this case, for the corresponding variational inequality, i.e. when $G\equiv g(x,t)$, in \cite{RodriguesSantos2015} it was shown that the problem is well-posed  and has similar stability properties, as in Section \ref{3.1}, with coefficients only in $L^2$.

Suppose that, for some $l\in\R$,
\begin{equation}\label{bcl}
\bs b\in L^2(Q_T)^d,\quad c\in L^2(Q_T)\ \text{ and }\ c-\frac12\nabla\cdot\bs b\ge l\ \text{ in } Q_T,
\end{equation}
and
\begin{equation}\label{fgu0}
f\in L^2(Q_T), \quad g\in W^{1,\infty}\big(0,T;L^\infty(\Omega)\big),\ g\ge\nu>0\ \text{ and }\ u_0\in\K_{g(0)}.
\end{equation}

\begin{theorem} {\em \cite{RodriguesSantos2015}} With the assumptions \eqref{bcl} and \eqref{fgu0}, there exists a unique strong solution
$$w\in L^\infty\big(0,T;W^{1,\infty}_0(\Omega)\big)\cap\C(\overline Q_T),\qquad \partial_t w\in L^2(Q_T),$$
to the variational inequality
\begin{equation}\label{ivbc}
\begin{cases}
w(t)\in\K_{g(t)},\ t\in(0,T),\quad w(0)=u_0,\vspace{0.5mm}\\
\displaystyle\int_\Omega \big(\partial_t w(t)+\bs b(t)\cdot\nabla w(t)+c(t)w(t)\big)(v-w(t))\vspace{0.5mm}\\
\qquad\qquad\qquad\displaystyle\hfill{	\ge \int_\Omega f(t)(v-w(t)),\quad\forall v\in\K_{g(t)},\ \text{ for a.e. }t\in(0,T).}
\end{cases}
\end{equation}\qed
\end{theorem}

The corresponding stationary problem for
\begin{equation}\label{ivbcinfty}
w_\infty\in\K_{g_\infty}:\quad\int_\Omega \big(\bs b_\infty\cdot\nabla u_\infty+c_\infty w_\infty\big)(v-w_\infty)\ge \int_\Omega f_\infty(v-w_\infty)\quad\forall v\in\K_{g_\infty}
\end{equation}
can be solved uniquely for $L^1$ data
\begin{equation}\label{bcinfty}
\bs b_\infty\in L^1(\Omega)^d,\quad c_\infty\in L^1(\Omega)\ \text{ and }\ c_\infty-\frac12\nabla\cdot\bs b_\infty\ge\mu\ \text{ in }\Omega,
\end{equation}
with
\begin{equation}\label{fginfty}
\bs f_\infty\in L^1(\Omega),\quad g_\infty\in L^\infty(\Omega),\  g_\infty\ge\nu>0.
\end{equation}and is the asymptotic limit of the solution of \eqref{ivbc}.

\begin{theorem}{\em \cite{RodriguesSantos2015}}
Under the assumptions \eqref{bcinfty} and \eqref{fginfty}, if
$$\int_t^{t+1}\int_\Omega\big(|f(\tau)-f_\infty|+|\bs b(\tau)-\bs b_\infty|+|c(\tau)-c_\infty|\big)dxd\tau\underset{t\rightarrow\infty}{\longrightarrow}0$$
and there exists $\gamma>\frac12$ such that, for some constant $C>0$,
$$\|g(t)-g_\infty\|_{L^\infty(\Omega)}\le\frac{C}{t^\gamma},\quad t>0,$$
then
$$w(t)\underset{t\rightarrow\infty}{\longrightarrow} w_\infty$$
where $w$ and $w_\infty$ are, respectively, the solutions of the variational inequality \eqref{ivbc} and \eqref{ivbcinfty}.\qed
\end{theorem}

\subsection{The quasi-variational inequality via compactness and monotonicity} 

The results in Section \ref{ivgde} are for scalar functions and  $L=\nabla$. As the arguments in the proof that $\partial_tu$ is a Radon measure do not apply to the vector cases, we consider the weak quasi-variational inequality  for a given $\delta\ge0$, for $\bs u=\bs u^\delta$,
\begin{equation}\label{iqve}
\left\{\begin{array}{l}
\bs u\in\K_{G[\bs u]},\vspace{1mm}\\
\displaystyle\int_0^T\langle\partial_t\bs v,\bs v-\bs u\rangle_p+\delta\int_{Q_T} \L_p\bs u\cdot\text{ L}(\bs v-\bs u)\ge\displaystyle\int_{Q_T}\bs f\cdot(\bs v-\bs u)\vspace{1mm}\\
\hspace{3cm}\displaystyle-\frac12\int_\Omega|\bs v(0)-\bs u_0|^2,
\forall\,\bs v\in \Y_p\text{ such that }
\bs v\in\K_{G[\bs u]},
\end{array}
\right.
\end{equation}
where $\langle\,\cdot\,,\,\cdot\,\rangle_p$ denotes the duality pairing between $\X_p'\times\X_p$.

\begin{theorem}\label{iqvfracat} Suppose that assumptions  \eqref{L}, \eqref{Xp}, \eqref{norm}, \eqref{gelfand} are satisfied and $\bs f\in L^2(Q_T)^m$, $\bs u_0\in\K_{G(\bs u_0)}$. Assume, in addition that	$G:\HHH\rightarrow L^1(Q_T)$ is a nonlinear continuous functional whose restriction to $\V_p$ is compact with values in $\C\big([0,T];L^\infty(\Omega)\big)$ and  $G(\HHH)\subset L_\nu^\infty(Q_T)$ for some $\nu>0$.
	
Then the quasi-variational inequality \eqref{iqve} has a weak solution 
$$\bs u\in\V_p\cap L^\infty\big(0,T;L^2(\Omega)^m\big).$$
\end{theorem}
\begin{proof} We give a brief idea of the proof. The details can be found, in a more general setting, in \cite{MirandaRodriguesSantos2018}.
	
Assuming first $\delta>0$, we consider the following family of approximating problems, defined for  fixed $\bs\varphi\in\HHH$, such that $\bs u_0\in\K_{G[\bs\varphi(0)]}$: 
to find $\bs u_{\varepsilon,\bs \varphi}$ such that $\bs u_{\varepsilon,\bs\varphi}(0) =\bs u_0$ and
\begin{multline}\label{app}
\langle\partial_t\bs u_{\varepsilon ,\bs\varphi}(t),\bs\psi\rangle_p
+\int_\Omega\big(\delta+ k_\varepsilon\big(|L \bs u_{\varepsilon ,\bs\varphi}(t)| -G[\bs\varphi](t)\big)\big) \L_p \bs u_{\varepsilon ,\bs\varphi}(t)\cdot\text{L}\bs\psi\\
=\int_\Omega\bs f(t)\cdot\bs\psi,\qquad\forall \bs\psi\in\X_p,\quad\text{for a.e. }t\in(0,T),
\end{multline}	
where $k_{\eps}:\R\to\R$ is an increasing continuous function such that
\begin{equation*}\label{kapa_eps}
k_\varepsilon(s) = 0 \ \ \text{ if } s\leq0,\qquad  k_\varepsilon(s)=e^\frac{s}{\varepsilon}-1\ \ \text{ if }0\leq s\leq \tfrac1\varepsilon,\qquad k_\varepsilon(s)=e^\frac{1}{\varepsilon^2}-1\ \ \text{ if } s\geq\tfrac1\varepsilon.
\end{equation*}
This problem has a unique solution $\bs u_{\varepsilon,\bs\varphi}\in \V_p$, with $\partial_t\bs u_{\varepsilon,\bs\varphi}\in\V_p'$.
Let  $S:\HHH\rightarrow\Y_p$ be the mapping that assigns to each $\bs \varphi\in \HHH$ the unique solution  $\bs u_{\varepsilon,\bs\varphi}$ of problem \eqref{app}. Considering the embedding $i:\Y_p\rightarrow\HHH$, then  $i\circ S$ is continuous, compact and we have a priori estimates which assures that there exists a positive $R$, independent of $\eps$, such that
$i\circ S(\HHH)\subset D_{R},$
where $D_{R}=\big\{\bs w\in\HHH:\|\bs w\|_\HHH\le R\big\}$. By Schauder's fixed point theorem, $i\circ S$ has a fixed point $\bs u_{\eps}$, which solves problem \eqref{app} with $\bs\varphi$ replaced by $\bs u_{\eps}$.

The sequence $\{\boldsymbol u_\eps\}_{_\eps}$ satisfies a priori estimates which allow us to obtain the limit $\boldsymbol u$ for subsequences in $\V_p\cap L^\infty\big(0,T;L^2(\Omega)^m\big)$. Another main estimate
$$
\|k_\eps(|L\bs u_{\eps}|-G[\bs u_{\eps}])\|_{L^1(Q_T)}\le C,$$
with $C$ a constant independent of $\eps$, yields $\boldsymbol u\in K_{G[\boldsymbol u]}$. 

Using $\bs u_\eps-\bs v$ as test function in \eqref{app} corresponding to a fixed point $\bs\varphi=\bs u_\eps$, with an arbitrary $\bs v\in\V_p\cap\K_{G[\bs u]}$, we obtain, after integration in $t\in(0,T)$ and setting $k_\eps=k_\eps\big(|\text{L}\bs u_\eps|-G[\bs u_\eps]\big)$:
\begin{multline}\label{solfrac}
\delta\int_{Q_T}\L_p\bs u_\eps\cdot\text{L}(\bs u_\eps-\bs v)\le\frac12\int_{\Omega}|\bs u_0-\bs v(0)|^2+\int_0^T\langle\partial_t\bs v,\bs v-\bs u_\eps\rangle_p\\
+\int_{Q_T}k_\eps\L_p\bs u_\eps\cdot\text{L}(\bs u_\eps-\bs v)-\int_{Q_T}\bs f\cdot(\bs v-\bs u_\eps).
\end{multline}

The passage to the limit $\bs u_\eps\underset{\eps\rightarrow0}{\lraup}\bs u$ in order to conclude that \eqref{iqve} holds for $\bs u$ is delicate and requires a new lemma, which proof can be found in \cite{MirandaRodriguesSantos2018}: given $\bs w\in\V_p$ such that $\bs w\in\K_{G[\bs w]}$ and $\bs z\in\K_{G[\bs w(0)]}$, we may construct a regularizing sequence $\{\bs w_n\}_n$ and a sequence of scalar functions $\{G_n\}_n$ satisfying  i) $\bs w_n\in L^\infty(0,T;\X_p)$ and $\partial_t\bs w_n\in L^\infty(0,T;\X_p)$,
ii) $\bs w_n\underset{n}{\longrightarrow}\bs w$ in $\V_p$ strongly,
iii) $\varlimsup_n\int_0^T\langle\partial_t\bs w_{ n}, \bs w_{ n}-\bs w\rangle_p\le0$ and
iv) $|L\bs w_n|\le G_n$, where $G_n\in\C\big([0,T];L^\infty(\Omega)\big)$ and $G_n\underset{n}{\longrightarrow}G[\bs w]$ in $\C\big([0,T];L^\infty(\Omega)\big)$.

If $\{\bs u_n\}_n$ is a regularizing sequence associated to $\bs u$ and $G[\bs u]$ then there exists a constant $C$ independent of $\eps$ and $n$ such that
\begin{multline*}\int_{Q_T}\!\!k_\eps\L_p\bs u_\eps \cdot\text{L}(\bs u _n-\bs u_\eps )\\
\le \int_{Q_T}\!\!k_\eps|\text{L}\bs u_\eps |^{p-1}\big(|\text{L}\bs u_n|-|\text{L}\bs u_\eps |\big)\le C\|G_n-G[\bs u_\eps ]\|_{L^\infty(Q_T)}\underset{n}{\longrightarrow}0,
\end{multline*}
by the compactness of the operator $G$.
For all $n\in\N$ we have, setting $\bs v=\bs u_\eps$,
\begin{multline*}\int_{Q_T}\delta\L_p\bs u_\eps\cdot\text{L}(\bs u_\eps-\bs u_n)\le\int_0^T\langle\partial_t\bs u_n ,\bs u_n-\bs u\rangle_p\\
+\int_{Q_T}\delta\L_p\bs u_\eps \cdot\text{L}(\bs u_n -\bs u)+\int_{Q_T}\!\!k_\eps\L_p\bs u_\eps \cdot\text{L}(\bs u _n -\bs u_\eps )-\int_{Q_T}\bs f\cdot(\bs u_n -\bs u),
\end{multline*}
concluding that
$$\varlimsup_{\eps\rightarrow0}\int_{Q_T}\delta\L_p\bs u_\eps \cdot\text{L}(\bs u_\eps -\bs u)\le0.$$
This operator is bounded, monotone and hemicontinuous and so it is pseudo-monotone and we get, using \eqref{solfrac},
$$\int_{Q_T}\delta\L_p\bs u\cdot\text{L}( \bs u-\bs v)\le\varliminf_{\eps\rightarrow0}\int_{Q_T}\delta\L_p\bs u_\eps \cdot\text{L}(\bs u_\eps -\bs v),\quad\forall\bs v\in \K_{G[\bs u]}$$
and the proof that $\bs u$ solves the quasi-variational inequality \eqref{iqve} is now easy, by using the well-known monotonicity methods (see \cite{Brezis1968} or \cite{ Lions1969}).

The proof for the case $\delta=0$ is more delicate and requires  taking the limit of diagonal subsequences of solutions $\{(\eps,\delta)\}_{_{\eps,\delta}}$ of \eqref{app} as $\eps\rightarrow0$ and as $\delta\rightarrow0$, in order to use the monotonicity methods to obtain a solution of \eqref{solfrac} in the degenerate case.
\qed\end{proof}

\begin{remark}
Two general examples for the compact operator $G:\V_p\rightarrow\C\big([0,T];L^\infty_\nu(\Omega)\big)$ in the form $G[\bs v]=g(x,t,\zeta(\bs v(x,t)))$, with $g\in\C(\overline Q_T\times\R^m)$, $g\ge\nu>0$, were given in \cite{MirandaRodriguesSantos2018}, namely with 
$$\zeta(\bs v)(x,t)=\int_0^t\bs v(x,s)K(t,s)ds,\qquad (x,t)\in\overline  Q_T,$$
with $K$, $\partial_t K\in L^\infty\big((0,T)\times(0,T)\big)$, or with $\zeta=\zeta(\bs v)$ given by the unique solution of the Cauchy-Dirichlet problem of a quasilinear parabolic scalar equation $\partial_t\zeta-\nabla\cdot a\big(x,t,\nabla\zeta)\big)=\varphi_0+\bs\psi\cdot\bs v+\bs\eta\cdot \text{L}\bs v\in L^p(Q_T)$, which has solutions in the H\"older space $\C^\lambda(\overline Q_T)$, for some $0<\lambda<1$, provided that $\bs v\in\V_p$, $p>\frac{d+2}d$ and $\varphi_0\in L^p(Q_T)$, $\bs\psi\in L^\infty(Q_T)^m$, $\bs\eta\in L^\infty(Q_T)^\ell$ are given.
\end{remark}
	
\begin{remark} Using the sub-differential analysis in Hilbert spaces, Kenmochi and co-workers have also obtained existence results in \cite{Ken2013} and \cite{KenmochiNiezgodka2016}  for evolutionary quasi-variational inequalities with gradient constraints under different assumptions.
\end{remark}

\subsection{The quasi-variational solution via contraction}

For the evolutionary quasi-variational inequalities and for nonlocal Lipschitz nonlinearities we can apply the Banach fixed point theorem in two different functional settings obtaining weak and strong solutions under certain conditions.

Let  $E$ be $L^2(Q_T)^m$ or $\V_p$ and
$$D_R=\{\bs v\in E:\|\bs v\|_E\le R\}.$$

For $\eta, M, \Gamma:\R\rightarrow\R^+$  increasing functions, let $\gamma:E\rightarrow\R^+$  be a functional satisfying\vspace{-3mm}
\begin{align}\label{gamma}
\nonumber	& 0<\eta(R_*)\le\gamma(\bs u)\le M(R_*)\quad \forall\,\bs u\in D_{R_*},\\
&|\gamma(\bs u_1)-\gamma(\bs u_2)|\le\Gamma(R_*)\|u_1-u_2\|_{E}\quad\ \forall\,\bs  u_1,\bs  u_2\in  D_{R_*},
\end{align}
for a sufficiently large  $R_*\in\R^+$.

\begin{theorem} For $p>1$ and $\delta\ge0$, suppose that the assumptions \eqref{L},  \eqref{Xp}, \eqref{norm} and \eqref{gelfand} are satisfied, $\bs f\in L^2(Q_T)^m$, 
$$G[\bs u](x,t)=\gamma(\bs u)\varphi(x,t), \quad (x,t)\in Q_T,$$
where $E=L^2(Q_T)^m$ and
$\gamma$ is a functional satisfying \eqref{gamma},  $\varphi\in\C\big([0,T];L^\infty_\nu(\Omega)\big)$, $\bs u_0\in\K_{G[\bs u_0]}$ and
\begin{equation}
\label{r*}
R_*=\sqrt{T+T^2e^T}\big(\|\bs f\|_{L^2(Q_T)^m}+\|\bs u_0\|_{L^2(\Omega)^m}\big).
\end{equation}
	
If
$$2\,R_*\Gamma(R_*)<\eta(R_*)$$
then the quasi-variational inequality \eqref{iqve}  has a unique weak  solution $\bs u\in \V_p\cap\C\big([0,T];L^2(\Omega)^m\big)$, which is also a strong solution $\bs u\in\V_p\cap H^1\big(0,T;L^2(\Omega)^m\big)$, provided $\varphi\in W^{1,\infty}(0,T;L^\infty(\Omega))$ with $\varphi\ge\nu>0$.
\end{theorem}
\begin{proof}
For any $R>0$ let
$S:D_R\rightarrow L^2(Q_T)^m$ be the mapping that, by Theorem \ref{teoIQV}, assigns to each $\bs v\in D_R$ the unique solution of the variational inequality \eqref{iv_weak} (respectively \eqref{ive}) with data $\bs f$, $G[\bs v]$ and $\bs u_0$. Denoting $\bs u=S(\bs v)=S(\bs f, G[\bs v],\bs u_0)$, using the stability result \eqref{stabilityef} with $\bs u_1=\bs u$ and $\bs u_2=0$ we have the estimate
\begin{multline}\label{est_r*}\|\bs u\|_{L^2(Q_T)^m}\le\sqrt T\|\bs u\|_{L^\infty(0,T;L^2(\Omega))}\\
\le \sqrt{T+T^2e^T}\big(\|\bs f\|_{L^2(Q_T)^m}+\|\bs u_0\|_{L^2(\Omega)^m}\big)=R_*,
\end{multline}
being $R_*$ fixed from now on.
For this choice of $R_*$ we have $S(D_{R_*})\subseteq D_{R_*}$.
	
For $\bs v_i\in D_{R_*}$, $i=1,2$ and $\bs u_i=S(\bs f,G[\bs v_i],\bs u_0)$, set $\mu=\frac{\gamma(\bs v_2)}{\gamma(\bs v_1)}$ which we may assume to be greater than 1.  Denoting $g=G[\bs v_1]=\gamma(\bs v_1)\varphi$, then $\mu u_1=S(\mu\bs f,\mu g,\mu\bs u_0)$, $\bs u_2=S(\bs f,\mu g,\bs u_0)$ and, using \eqref{stabilityef}, we have
\begin{align*}
\nonumber	\|S(\bs v_1)-S(\bs v_2)\|_{L^2(Q_T)^m}&\le \|\bs u_1-\mu\bs u_1\|_{L^2(Q_T)^m}+\|\mu\bs u_1-\bs u_2\|_{L^2(Q_T)^m}\\
\nonumber	&\le (\mu-1)\|\bs u_1\|_{L^2(Q_T)^m}+(\mu-1)R_*\le 2(\mu-1)R_*.
\end{align*}
But
$$\mu-1=\frac{\gamma(\bs v_2)-\gamma(\bs v_1)}{\gamma(\bs v_1)}\le \frac{\Gamma(R_*)}{\eta(R_*)}\|\bs v_1-\bs v_2\|_{L^2(Q_T)^m}$$
and consequently $S$ is a contraction as long as
	$$\frac{2\,R_*\Gamma(R_*)}{\eta(R_*)}<1.$$
\end{proof}

\begin{remark} These results are new. In particular, the one with $\varphi$ more regular gives the existence and uniqueness of the strong solution $\bs u\in \V_p\cap H^1\big(0,T;L^2(\Omega)^m\big)$ to the quasi-variational inequality \eqref{iqve} and therefore also satisfies $\bs u(t)\in\K_{G[\bs u(t)]}$ and \eqref{ive} with $g=G[\bs u(t)]$,
$$\int_{\Omega}\partial_t\bs u(t)\cdot(\bs w-\bs u(t))+\delta\int_\Omega\L_p\bs u(t)\cdot \text{L}(\bs w-\bs u(t))
\ge\int_{\Omega}\bs f(t)\cdot(\bs w-\bs u(t)),$$ for all $\bs w\in \K_{G[\bs u](t)}$, a.e. $t\in(0,T)$.
\end{remark}

\begin{theorem}\label{iqvecont}
For $1<p\le2$ and $\delta>0$, suppose that the assumptions \eqref{L},  \eqref{Xp}, \eqref{norm} and \eqref{gelfand} are satisfied, $\bs f\in L^2(Q_T)^m$, 
$$G[\bs u](x,t)=\gamma(\bs u)\varphi(x,t),\quad (x,t)\in Q_T,$$
where $E=\V_p$, $\gamma$ is a functional satisfying
\eqref{gamma}, $\varphi\in \C\big([0,T];L^\infty_\nu(\Omega)\big)$  and $\bs u_0\in\K_{G[\bs u_0]}$. Then, the quasi-variational inequality \eqref{iqve} has a unique weak  solution $\bs u\in\V_p\cap\C\big([0,T];L^2(\Omega)^m\big)$, provided that
\begin{equation}\label{Banach_ce}
\rho\,\Gamma(R_p)<\eta(R_p),
\end{equation}
where
$$\rho=2R_p+(2-p)\,\big(2\, M(R_p)\,\|\varphi\|_{L^\infty(Q_T)}\big)^{2-p}|Q_T|^\frac{2-p}{p} (R_p)^{p-1}$$	 
and 
$$R_p=\Big(\tfrac{1+T+T^2e^T}{2\delta}\big(\|\bs  f\|_{L^2(Q_T)^m}^2+\|\bs u_0\|_{L^2(\Omega)^m}^2\big)\Big)^\frac1p,$$	
which is also a strong solution in $\V_p\cap H^1\big(0,T;L^2(\Omega)^m\big)$ if, instead, we have $\varphi\in W^{1,\infty}(0,T;L^\infty(\Omega))$ with $\varphi\ge\nu>0$.
\end{theorem}
\begin{proof} For $R>0$ let  $S:D_R\rightarrow\V_p$ be defined by $\bs u=S(\bs v)=S(\L_p, \bs f, g,\bs u_0)$, the unique strong solution of the variational inequality \eqref{ive}, with the operator $\L_p$ and data $(\bs f,g,\bs u_0)$, where $g=G[\bs v]$. Taking $\bs w=\bs0$ in \eqref{ive} and using the estimate \eqref{est_r*} we have the a priori estimate
\begin{align*}
2\,\delta\|\bs u\|_{\V_p}^p&\le\big(\|\bs  f\|_{L^2(Q_t)^m}^2+\|\bs u\|^2_{L^2(Q_T)^m}+\|\bs u_0\|_{L^2(\Omega)^m}^2\big)\\
&\le(1+T+T^2e^T)\big(\|\bs  f\|_{L^2(Q_t)}^2+\|\bs u_0\|_{L^2(\Omega)}^2\big)
\end{align*}
and therefore
\begin{equation}
\label{raio}
\|\bs u\|_{\V_p}\le\Big(\tfrac{1+T+T^2e^T}{2\delta}\big(\|\bs  f\|_{L^2(Q_T)^m}^2+\|\bs u_0\|_{L^2(\Omega)^m}^2\big)\Big)^\frac1p=R_p.
\end{equation}
	
Given $\bs v_i\in D_{R_p}$, $i=1,2$, let $\bs u_i=S(\L_p,\bs f,\gamma(\bs v_i)\varphi,\bs u_0)$ and set  $\mu=\frac{\gamma(\bs v_2)}{\gamma(\bs v_1)}$, assuming $\mu>1$ . 
	
Setting $g=G[\bs v_1]=\gamma(\bs v_1)\varphi$, observe that   $\mu\bs u_1=S(\mu^{2-p}\L_p,\mu\bs f,\mu g,\mu \bs u_0)=\bs z_1$ and $\bs z_2=S(\L_p,\mu\bs f,\mu g,\mu \bs u_0)$, we get
\begin{equation*}
\|\bs u_1-\bs u_2\|_{\V_p}\le \|\bs u_1-\bs z_1\|_{\V_p}+\|\bs z_1-\bs z_2\|_{\V_p}+\|\bs z_2-\bs u_2\|_{\V_p}.
\end{equation*}
By \eqref{raio} and the continuous dependence result \eqref{stabilityeff},
\begin{equation} \label{z1z2antesantes}
\|\bs u_1-\bs z_1\|_{\V_p}=(\mu-1)\|\bs u_1\|_{\V_p}\le(\mu-1)R_p\ \ \text{ and }\ \ \|\bs z_2-\bs u_2\|_{\V_p}=(\mu-1)R_p.
\end{equation}	
	
Since $\bs z_1,\bs z_2\in\K_{\mu g}$, we can use them as test functions in the variational inequality \eqref{ive} satisfied by the other one. Then
\begin{multline*}
\frac12\int_{\Omega}|\bs z_1(t)-\bs z_2(t)|^2+\int_{Q_T}|\text{L}(\bs z_1-\bs z_2)|^2\big(|\text{L}\bs z_1|+|\text{L}\bs z_2|\big)^{p-2}\\
\le(\mu^{2-p}-1)\int_{Q_T}\text{L}_p\bs z_1\cdot\text{L}(\bs z_1-\bs z_2)
\end{multline*}
and, by the H\"older inverse inequality,
\begin{multline*}
\frac12\int_{\Omega}|\bs z_1(t)-\bs z_2(t)|^2+\|\text{L}(\bs z_1-\bs z_2)\|^2_{L^p(Q_T)^\ell}\Big(\int_{Q_T}\big(|\text{L}\bs z_1|+|\text{L}\bs z_1|\big)^{p}\Big)^{\frac{p-2}{p}}\\
\le(\mu^{2-p}-1)\|\text{L}\bs z_1\|_{L^p(Q_T)^\ell}^{p-1}\|\text{L}(\bs z_1-\bs z_2)\|_{L^p(Q_T)^\ell}.
\end{multline*}
But
$$\Big(\int_{Q_T}\big(|\text{L}\bs z_1|+|\text{L}\bs z_1|\big)^{p}\Big)^{\frac{p-2}{p}}\ge(2\, M(R_p)\,\|\varphi\|_{L^\infty(Q_T)})^{p-2}|Q_T|^\frac{p-2}{p}$$
and $\mu^{2-p}-1\le(2-p)(\mu-1)$, so
\begin{equation}\label{estz1z2}
\|\bs z_1-\bs z_2\|_{\V_p}\le (\mu-1)\,(2-p)\,\big(2\, M(R_p)\,\|\varphi\|_{L^\infty(Q_T)}\big)^{2-p}|Q_T|^\frac{2-p}{p} (R_p)^{p-1}.
\end{equation}
	
From \eqref{z1z2antesantes} and \eqref{estz1z2}, we obtain
\begin{multline}
\|S(\bs v_1)-S(\bs v_2)\|_{\V_p}\le(\mu-1)\Big(2R_p\\
+(2-p)\big(2\, M(R_p)\,\|\varphi\|_{L^\infty(Q_T)}\big)^{2-p}|Q_T|^\frac{2-p}{p} (R_p)^{p-1}\Big).
\end{multline}
	
Defining
$$\rho=2R_p+(2-p)\,\big(2\, M(R_p)\,\|\varphi\|_{L^\infty(Q_T)}\big)^{2-p}|Q_T|^\frac{2-p}{p} (R_p)^{p-1}$$
we get, with $\Gamma=\Gamma(R_p)$ and $\eta=\eta(R_p)$,
$$\|S(\bs v_1)-S(\bs v_2)\|_{\V_p}\le \frac{\rho\,\Gamma}{\eta}\|\bs v_1-\bs v_2\|_{\V_p}$$
and $S$ is a contraction if 
$\rho\,\Gamma<\eta,$ which fixed point $\bs u\in\V_p\cap H^1\big(0,T;L^2(\Omega)^m\big)$ is the strong solution of the quasi-variational inequality.
	
In the case of $\varphi\in\C\big([0,T];L^\infty_\nu(\Omega)\big)$, the solution map $S$ of Theorem~\ref{teoIQV} only gives a weak solution $\bs u\in\V_p\cap\big([0,T];L^2(\Omega)^m\big)$, which is a contraction exactly in the same case as \eqref{raio}. The proof is the same, since the continuous dependence estimate \eqref{estz1z2} still holds for weak solutions of the variational inequality as in Theorem~\ref{cdde}. 
\qed\end{proof}

\begin{remark} These results apply to nonlocal dependences on the derivatives of $\bs u$ as well, since $\varphi$ is Lipschitz continuous on $\V_p$. The part corresponding to weak solutions is new, while  the one for strong solutions extends \cite[Theorem 3.2]{HintermullerRautenberg2017}. This work considers strong solutions in the abstract framework of \cite{Lions1969}, which also include obstacle problems, it is aimed to numerical applications, but requires stronger restrictions on $\varphi$.
\end{remark}
	
\subsection{Applications}

\begin{example} {\bf The dynamics of the sandpile}$ $
	
\noindent Among the continuum models for granular motion, the one proposed by Prigozhin (see \cite{P86}, \cite{Prigozhin1994} and \cite{Pri1996-1}) for the pile surface $u=u(x,t)$, $x\in\Omega\subset\R^2$, growing on a rigid support $u_0=u_0(x)$,  satisfying the repose angle $\alpha$ condition, i.e., the surface slope $|\nabla u|$ cannot exceed $k=\tan\alpha>0$ nor the support slope $|\nabla u_0|$. This leads to the implicit gradient constraint
\begin{equation}\label{KG0}
|\nabla u(x,t)\le G_0[u](x,t)\equiv \begin{cases}
k&\text{ if }u(x,t)>u_0(x)\\
k\vee|\nabla u_0(x)|&\text{ if }u(x,t)\le u_0(x).
\end{cases}
\end{equation}

Following \cite{PrigozhinZaltzman2003}, the pile surface dynamics is related to the thickness $v=v(x,t)$ of a thin surface layer of rolling particles and  may be described by
\begin{equation}\label{sand}
\partial_tu+v\big(1-\tfrac{|\nabla u|^2}{k}\big)\quad\text{and}\quad\eps\partial_tv-\eta\nabla\cdot(v\nabla u)=f-v\big(1-\tfrac{|\nabla u|^2}{k}\big),
\end{equation}
where $\eps\sim0$ is the ratio of the thickness of the rolling grain layer and the pile size, $\eta>0$, is a ratio characterizing the competition between rolling and trapping of the granular material, and $f$ the source intensity, which is positive for the growing pile, but may be zero or negative for taking erosion effects into account. Assuming $v(x,t)=\overline v>0$, from \eqref{sand} we obtain
$$\partial_tu-\delta\Delta u=f\ \text{ if }\ |\nabla u|<G_0,$$
where $\delta=\eta\overline v>0$ may account for a small rolling of sand and hence some surface diffusion below the critical slope, or no surface flow if $\delta=0$.

Assuming an homogeneous boundary condition, which means the sand may fall out of $\partial\Omega$, and the initial condition below the critical slope, i.e., $|\nabla u_0|\le k$, the pile surface $u=u^\delta(x,t)$, $\delta\ge 0$, is the unique solution of the scalar variational inequality \eqref{ive} with $L=\nabla$, $p=2$ and $g(t)\equiv k$, provided we prescribe
$f\in L^2(Q_T)$. We observe that, by comparison of $u_1=u^\delta$ with $\delta>0$ and $u_2=u^0$ with $\delta=0$, as in Theorem~\ref{cdde}, we have the estimate
$$\int_\Omega|u^\delta-u^0|^2(t)\le 2\delta\int_0^t\int_\Omega|\nabla u^\delta||\nabla(u^\delta-u^0)|\le 4\delta k^2|Q_t|,\quad0<t<T.$$

We can also immediately apply for $t\rightarrow\infty$ the asymptotic results of Theorem~\ref{asympttt} for $\delta>0$ and Theorem~\ref{asympt22} for $\delta=0$.

Moreover, if $\delta=0$ in the case of the growing pile with $f(x,t)=f(x)\ge0$ it was observed in \cite{CannarsaCardaliaguetSinestrari2009} not only that, if $t>s>0$
\begin{equation}
\label{stabc}
u_0(x)\le u(x,s)\le u(x,t)\le u_\infty(x)=\lim_{t\rightarrow\infty}u(x,t)\le kd(x),\ x\in\Omega,
\end{equation}
where $d(x)=d(x,\partial\Omega)$ is the distance function to the boundary, but the limit stationary solution is given by
$$u_\infty(x)=u_0(x)\vee u_f(x),\quad x\in\Omega,$$
where $\displaystyle u_f(x)=\max_{y\in\supp f}\big(d(y)-|x-y|\big)^+$, $x\in\Omega$. This model has also a very interesting property of the finite time  stabilization of the sandpile, provided that $f$ is positive in a neighborhood of the ridge $\Sigma$ of $\Omega$, i.e. the set of points $x\in\Omega$ where $d$ is not differentiable (see \cite[Theorem 3.3]{CannarsaCardaliaguetSinestrari2009}): there exists a time $T<\infty$ such that, for any $u_0\in\K_k=\{v\in H^1_0(\Omega):|\nabla v|\le k\}$,
\begin{equation}\label{finite}
u(x,t)=kd(x),\quad\forall t\ge T,
\end{equation}
provided $\exists\, r>0: f(x)\ge r\text{ a.e. }x\in B_r(y)$, for all $y\in\Sigma.$

Similar results were obtained in \cite{RodriguesSantos2015} for the transported sandpile problem, for $u(t)\in \K_k$, such that
\begin{equation}\label{sandb}
\int_\Omega\big(\partial_tu(t)+\bs b\cdot\nabla u(t)-f(t)\big)(v-u(t))\ge0,\quad\text{ a.e. }t\in(0,T),
\end{equation}
for all $v\in\K_k$, with $\bs b\in\R^2$, $\partial\Omega\in\C^2$, $f=f(t)\ge0$ nondecreasing and $f\in L^\infty(0,\infty)$, which also satisfies \eqref{stabc}. Moreover, it was also shown in \cite{RodriguesSantos2015} that $u(t)$ equivalently solves \eqref{sandb} for the double obstacle problem, i.e. with $\K^\wedge_\vee=\{v\in H^1_0(\Omega):-kd(x)\le v(x)\le kd(x),\ x\in\Omega\}$ and, moreover, has also the finite time stabilization property \eqref{finite} under the additional assumptions $\bs b\cdot\nabla u_0\le f(t)$ in $\{x\in\Omega:-kd(x)<u_0(x)\}$ for $t>0$ and $\displaystyle\liminf_{t\rightarrow\infty} f(t)>|\bs b|+2k\|d\|_{L^\infty(\Omega)}$.

It should be noted that if we replace  $\K_k$ by the solution dependent convex set $\K_{G_0[u]}$, with $G_0$ defined in \eqref{KG0}, to solve the corresponding quasi-variational inequality \eqref{sandb}, even with $\bs b\equiv\bs 0$ or with an additional $\delta$-diffusion term is an open problem since the operator $G_0$ is not continuous in $u$. Recently, in \cite{BarrettPrigozhin2013-2}, Barrett and Prigozhin succeeded to construct, by numerical analysis methods, approximate solutions, including numerical examples, that converge to a quasi-variational solution of \eqref{sandb} without transport ($\bs b\equiv\bs0$), for fixed $\eps>0$, with the continuous operator $G_\eps:\C(\overline\Omega)\rightarrow \C(\overline\Omega)$ given by
$$G_\eps[u](x,t)=\begin{cases}
k&\text{ if }u(x)\ge u_\eps(x)+\eps,\vspace{0,5mm}\\
k_\eps(x)+(k-k_\eps(x))\frac{u(x)-u_\eps(x)}{\eps}&\text{ if }u_\eps(x)\le u(x)<u_\eps(x)+\eps,\vspace{0,5mm}\\
k_\eps(x)\equiv k\vee|\nabla u_\eps(x)|&\text{ if }u(x)<u_\eps(x),
\end{cases}$$
where $u_\eps\in\C^1(\overline{\Omega})\cap W^{1,\infty}_0(\Omega)$ is an approximation of the initial condition $u_0\in W^{1,\infty}_0(\Omega)$. We observe that the existence of a quasi-variational solution of \eqref{sandb} with this $G_\eps$ is also guaranteed by Theorem~\ref{arma} or Corollary~\ref{coro}.
\end{example} 
	
\begin{example} {\bf An evolutionary electromagnetic heating problem} \cite{MirandaRodriguesSantos2012}		$ $
	
\noindent We consider now an evolutionary case of the Example 6 for the magnetic field $\bs h=\bs h(x,t)$ of a superconductor, which threshold may depend of a temperature field $\vartheta=\vartheta(x,t)$, $(x,t)\in Q_T$, subjected to a magnetic heating. This leads to the quasi-variational weak formulation
\begin{equation}\label{iqvhh}
\begin{cases}
\bs h\in \K_{j(\vartheta(\bs h))}\subset\V_p,\vspace{0.5mm}\\
\displaystyle\int_0^T\langle\partial_t \bs w,\bs w-\bs h\rangle_p+\delta\int_{Q_T}|\nabla\times\bs h|^{p-2}\nabla\times(\bs w-\bs h)\ge\int_{Q_T}\bs f\cdot(\bs w-\bs h)\vspace{0.5mm}\\
\displaystyle	\qquad\qquad\qquad\hfill{	-\frac12\int_{\Omega}|\bs w(0)-\bs h_0|^2\quad\forall\bs w\in\Y_p\ \text{ such that }\bs w\in\K_{j(\vartheta(\bs h))}}
\end{cases}
\end{equation}
coupled with a Cauchy-Dirichlet problem for the heat equation
\begin{align}\label{eqtheta}
\nonumber &\partial_t\vartheta-\Delta\vartheta=\eta+\bs\zeta\cdot\bs h+\bs\xi\cdot\nabla\times\bs h\quad\text{ in }Q_T,\\
&\vartheta=0\text { on }\partial\Omega\times(0,T),\quad \vartheta(0)=\vartheta_0\ \text{ in }\Omega.
\end{align}
		
Here, for a.e. $t\in(0,T)$, the convex set depends on $\bs h$ trough $\vartheta$ and is given, for some $j=j(x,t,\vartheta)\in\C(\overline Q_T\times\R)$, $j\ge\nu>0$ by
\begin{equation}\label{kapah}
\K_{j(\vartheta(t))}=\big\{\bs w\in\X_p:|\nabla\times\bs w|\le j(\theta(t))\text{ in }\Omega\big\}
\end{equation}
where $\X_p$ is given by \eqref{rot_normal0} or \eqref{rot_tangencial0}.
		
If we give $\vartheta_0\in H^1_0(\Omega)\cap\C^\alpha(\overline{\Omega})$, $ \eta\in L^p(Q_T)$ and $\bs \zeta$, $\bs \xi\in L^\infty(Q_T)^3$, the solution map that, for $p\ge\frac52$, associates to each $\bs h\in\V_p$, the unique solution $\vartheta\in L^2\big(0,T;H^1_0(\Omega)\big)\cap\C^\lambda(\overline Q_T)$, for some $0<\lambda<1$, is continuous and compact as a linear operator from $\V_p$ in $\C(\overline Q_T)$. Therefore, with $\bs f\in L^2(Q_T)^3$ and $\vartheta_0\in\K_{j(\vartheta_0)}$, Theorem \ref{iqvfracat} guarantees the existence of a weak solution $(\bs h,\vartheta)\in\big(\V_p\cap L^\infty\big(0,T;L^2(\Omega)^3)\big)\times \big(L^2\big(0,T;H^1_0(\Omega)\big)\cap\C^\lambda(\overline Q_T)\big)$ to the coupled problem \eqref{iqvhh}-\eqref{eqtheta}.
		
We observe that, if the threshold $j$ is independent of $\vartheta$, the problem becomes variational and admits not only weak but also strong solutions, by Theorem \ref{teoIQV}.
		
However, if we set a direct local dependence of the type $j=j(|\bs h|)$, as in Example \ref{6}, the problem is open in vectorial case.
		
Nevertheless, if the domain $\Omega=\omega\times(-R,R)$, with $\omega\subset\R^2$, $\partial\omega\in\C^{0,1}$ and the magnetic field has the form $\bs h=(0,0,u(y,t))$, $y\in\omega$, $0<t<T$, the critical-state superconductor model has a longitudinal geometry, where $u$ satisfies the scalar quasi-variational inequality \eqref{iqvpisaarma} with $\bs \Phi\equiv \bs 0$ and Theorem \ref{arma} provides in this case the existence of a strong solution $u\in\C(\overline\omega\times[0,T])\cap\K_{j(|\bs h|)}\cap W^{1,\infty}\big(0,T;M(\Omega)\big)$, with $j\in\C\big(\R;L^\infty_\nu(\omega)\big)$, for $\delta\ge0$. The case $\delta>0$ was first given in \cite{RodriguesSantos2000} and $\delta=0$ in \cite{RodriguesSantos2012}.\qed
\end{example}
		
\begin{example} {\bf Stokes flow for a thick fluid}$ $
	
\noindent The case where $\bs u=\bs u(x,t)$ represents the velocity field of an incompressible fluid in a limit case of a shear-thickening viscosity has been considered in \cite{ReyesStadler2014}, \cite{Rodrigues2014} and \cite{MirandaRodrigues2016} by using variational inequalities. Those works consider a constant or variable positive threshold on the symmetric part of the velocity field $L=D $. Here we consider the more general situation of a nonlocal dependence on the total energy of displacement
\begin{equation}\label{iqvfluidos}
|D \bs u(x,t)|\le G[\bs u(x,t)]=\varphi(x,t)\Big(\eta+\delta\int_{Q_T}|D \bs u|^2\Big),\  x\in\Omega\subset\R^d,\ t\in(0,T),
\end{equation}
for given $\delta,\eta>0$, $\varphi\in W^{1,\infty}\big(0,T;L^\infty(\Omega)\big)$, $\varphi\ge\nu>0$.

We set $\X_2=\big\{\bs w\in H^1_0(\Omega)^d:\nabla\cdot\bs w=0\big\}$, $d=2,3$, which is an Hilbert space for $\|D \bs w\|_{L^2(\Omega)^{d^2}}$ compactly embedded in $\mathbb H=\big\{\bs w\in L^2(\Omega)^d:\nabla\cdot\bs w=0\big\}$. Defining $\K_{G[\bs u](t)}$ for each $t\in(0,T)$ by \eqref{kapa} and giving $\bs f\in L^2(Q_T)^d$ and $\bs u_0\in\X_2$ satisfying \eqref{iqvfluidos} at $t=0$, i.e. $\bs u_0\in\K_{G[\bs u_0]}$, in order to apply Theorem~\ref{iqvecont}, we set
$$R_2=\tfrac{1+T+T^2e^T}{2\delta}\big(\|\bs f\|_{L^2(Q_T)^d}+\|\bs u_0\|_{L^2(\Omega)^d}\big)=\tfrac\rho2.$$

The nonlocal functional satisfies \eqref{gamma} with $E=L^2\big(0,T;\X_2)\big)=\V_2$ and $T=\delta\rho$, since we have
\begin{align*}
|\gamma(\bs u_1)-\gamma(\bs u_2)|&=\delta\Big|\int_{Q_T}|D\bs u_1|^2-|D\bs u_2|^2\Big|=\delta\Big|\int_{Q_T}(D\bs u_1-D\bs u_2)\cdot(D\bs u_1+D\bs u_2)\Big|\\
&\le \delta\rho\Big(\int_{Q_T}|D\bs u_1-D\bs u_2|^2\Big)^\frac12,\quad\text{ for }\bs u_1,\bs u_2\in D_{R_2}.
\end{align*}

Hence, by Theorem~\ref{iqvecont}, if $\delta\rho^2<\eta $, i.e. if
$$(1+T+T^2e^T)^2\big(\|\bs f\|_{L^2(Q_T)^d}+\|\bs u_0\|_{L^2(\Omega)^d}\big)<\tfrac\eta\delta,$$
there exists a unique strong solution $\bs u\in\V_2\cap H^1\big(0,T;L^2(\Omega)^d\big)\cap\K_{G[\bs u]}$, with $\bs u(0)=\bs u_0$, satisfying the quasi-variational inequality
$$\int_\Omega\partial_t\bs u(t)\cdot(\bs w-\bs u(t)))+\delta\int_\Omega D\bs u(t)\cdot D(\bs w-\bs u(t))\ge\int_\Omega\bs f(t)\cdot(\bs w-\bs u(t)),$$
for all $\bs w\in\K_{G[\bs u](t)}$ and a.e. $t\in(0,T)$.

This result can be generalized to the Navier-Stokes flows, i.e. with convection (see \cite{RodriguesSantos2018}).
\end{example}

\end{document}